\documentclass[12pt]{article}
\usepackage{times}
\usepackage{lmodern}
\usepackage{color}
\usepackage{amsfonts}
\usepackage{graphicx}
\usepackage[centertags]{amsmath}
\usepackage{amssymb}
\usepackage{amsthm}
\usepackage{newlfont}
\usepackage{mathrsfs}
\usepackage{hyperref}
\usepackage{cases}
\usepackage{amsmath}
\usepackage{bm}

\providecommand{\emailauthor}[2]{}
\providecommand{\corref}[1]{}

\theoremstyle{plain}
\newtheorem{thm}{Theorem}[section]

\newtheorem{lem}[thm]{Lemma}
\newtheorem{prop}[thm]{Proposition}

\theoremstyle{definition}
\newtheorem{defn}{Definition}[section]
\newtheorem{ass}{Assumption}[section]
\newtheorem{rmk}{Remark}[section]

\makeatletter\@addtoreset{equation}{section} \makeatother

\allowdisplaybreaks

\begin{document}

\title{The Backward Stochastic Partial Differential Integral
Equations: Solvability and Comparison Principle\thanks{This paper is supported by the National Key R\&D Program of China (No. 2022YFA1006101), the National Natural Science Foundation of China (Nos. 12271158 and 12371445), and the State Key Laboratory of Cryptography and Digital Economy Security, Shandong University (No. KFZD2505).}}

\date{}

\author{Qingxin Meng\thanks{Department of Mathematical Sciences, Huzhou Normal University, Zhejiang 313000, China. Email: mqx@huznu.edu.cn}\ \ and\ \
  Qi  Zhang\thanks{Corresponding author. School of Mathematical Sciences and Laboratory of Mathematics for Nonlinear Science, Fudan University, Shanghai 200433, China; State Key Laboratory of Cryptography and Digital Economy Security, Shandong University, Jinan 250100, China. Email: qzh@fudan.edu.cn.}
}

\maketitle

\begin{abstract}
The paper is concerned with the well-posedness of backward stochastic partial differential equations with jumps, also called backward stochastic partial differential integral equations. We start from the proof for the existence and uniqueness of solution to backward stochastic evolution equation with jump in the Gelfand triple framework. Then the well-posedness of both weak solution and strong solution to backward stochastic partial differential integral equation is obtained with the Gelfand triple replaced by specific Sobolev spaces. Finally, the comparison principle for backward stochastic partial differential integral equation is proved, which has potential applications in financial mathematics.
\end{abstract}

\textbf{Keywords}: backward stochastic partial differential integral equation, jump process, backward stochastic evolution equation, weak solution, strong solution, comparison principle.

\textbf{MSC2020}: 60H15, 60H20.

\section{Introduction}

The backward stochastic partial differential equation (BSPDE) is a backward stochastic equation which involves both stochastic noise and a partial differential operator. It was introduced by Bensoussan \cite{be1, be2} as the adjoint equation arising in the stochastic optimal control problem with partial observation, and has since been widely applied to financial mathematics, optimal control, stochastic filtering and physical modeling.

The well-posedness and regularity for such terminal value given stochastic equations is a challenging problem since the birth of BSPDEs.
%In 1992, the  estimates of the solutions to a class of BSPDEs in Sobolev space were obtained in Zhou \cite{zho} by a finite dimensional approximation (Galerkin method) and dual analysis of stochastic partial differential %equations, in which the high regularity requirements for coefficients was needed.
Zhou \cite{zh} in 1992 studies a class of linear BSPDEs under degenerate and non-degenerate conditions, respectively, in which the high regularity of coefficients is needed in the non-degenerate case and the equation form is special without the partial differential operator with respect to the solution in diffusion in the degenerate case. Du and Meng \cite{du-me} revisits linear non-degenerate BSPDEs and relaxes the high regularity of coefficients to a certain extent. After that many important progresses on non-degenerate BSPDE emerged in 2010s. Du and Tang \cite{du-ta} establishes the theory for the existence and uniqueness of strong solutions to linear and semilinear BSPDEs on bounded domains with $C^2$ boundaries; Du, Qiu and Tang \cite{du-qi-ta} investigates the $L^p$ theory and comparison principle for semilinear BSPDEs in the whole space; Qiu and Wei \cite{qi-we} studies the well-posedness and regularity of the reflected semilinear BSPDEs; Tang and Wei \cite{ta-we} obtains the solvability of the Cauchy problem for linear and semilinear BSPDEs in H\"{o}lder spaces, to name but a few. As for the degenerate BSPDEs, since the degeneracy brings much difficulty in the estimates of solutions, the results on the regularities of degenerate BSPDEs are relatively fewer than non-degenerate equations, and their regularities, Feynman-Kac formulas and  applications to mathematical finance as a stochastic Black-Scholes formula can be found in Tang \cite{ta}, Ma and Yong \cite{ma-yo1, ma-yo2}, Hu, Ma, and Yong \cite{hu-ma-yo}, Du, Tang and Zhang \cite{du-ta-zh}, Du and Zhang \cite{du-zh}, Horst, Qiu and Zhang \cite{ho-qi-zh}, Yang, Zhang and Zhang \cite{ya-zh-zh}, etc.

On the other hand, the existing results on BSPDEs driven by jump processes, i.e. backward stochastic partial differential integral equations (BSPDIEs), are very few. There are two well-known existing literatures. One is {\O}ksendal, Proske and Zhang \cite{ok-pr-zh} which studies a special type of non-degenerate semilinear BSPDEs with jumps without the partial differential operator with respect to the solution in diffusion. The other is Chen and Tang \cite{ch-ta} which investigates degenerate semilinear BSPDEs with jumps by the inverse mapping of stochastic flows with a high regularity requirement for coefficients.

Besides, there are also some studies on non-degenerate stochastic Hamilton-Jacobi-Bellman (HJB) equations (fully nonlinear BSPDEs), such as Peng \cite{pe}, Qiu \cite{qi} and Meng, Dong, Shi and Tang \cite{me-do-sh-ta}. Among them, \cite{me-do-sh-ta} is concerned with stochastic HJB equation driven by a jump process, but it only discusses weak solutions and does not explore higher regularity of solutions. In fact, the solvability of strong solution to stochastic HJB equation in a general form, even without jump process, is a long-lasting open problem. %It also struggles to effectively handle complex scenarios with non-smooth coefficients. Moreover, the
%$H^1$ space framework for weak solutions exhibits significant limitations in providing a mathematical %foundation for smooth control strategies.

We notice that most of the existing literatures on BSPDIEs suppose some special conditions, such as linear equations, specific L\'{e}vy process drivers, high regularity of coefficients. Moreover, most of well-posedness results focus on weak solutions, or smooth solutions for BSPDIEs in special forms. The results on the existence and uniqueness of strong solutions to BSPDIEs are less known. But the studies on BSPDIEs are theoretically and practically important. For instance, in the study of Pontryagin maximum principle, if the control problem is driven by a L\'{e}vy process, its corresponding adjoint equation is a BSPDIE. Needless to say, the models with jump perturbations in financial mathematics capture real-world market and risk dynamics better, and so BSPDIEs often appear in financial models involving random coefficients and jump processes. For example, in the European option pricing problem whose dynamics of underlying assets is path-dependent and discontinuous, the Black-Scholes equations are BSPDIEs.

Just under this background, our paper is concerned with the existence and uniqueness of strong solutions to BSPDIEs driven by jump processes with random coefficients. Compared to the jump-free case, the jump process brings a non-local integral operator and needs energy estimates and compactness analysis to overcome technical difficulties. For this purpose, we start from proving the well-posedness of solution to backward stochastic evolution equation (BSEE) with jump in the Gelfand triple framework. Then based on variational method, the well-posedness of both weak solution and strong solution in Sobolev spaces
%, i.e. in divergence form and non-divergence form,
is obtained with the Gelfand triple replaced by specific Sobolev spaces. %To be specific, we define the precise solution spaces, %establish a priori estimates by introducing energy inequalities, and prove the existence, uniqueness, and %stability of strong solutions by leveraging Lipschitz continuity conditions and fixed-point methods.
Unlike most existing literatures which require smooth enough assumptions for coefficients or special restrictions on jump terms, our work relaxes the requirements for coefficients and jump terms, and hence our results are applicable to more equations, such as those with degenerate diffusion, non-local jump and lower regularity of terminal value. % and non-homogeneous terms. %with the common Lipschitz and non-degenerate conditions,
%thus adapting to a wider range of application scenarios.
Moreover, the regularity improvements from weak solutions to strong solutions are studied in our paper. To be specific, based on localization techniques and parameter continuation methods, we upgrade the weak solutions to strong solutions by improving the regularities of solutions. To the best of our knowledge, this is the first discussion on regularity improvement of BSPDIEs.

The rest of this paper is organized as follows. In Section 2 we establish the existence and uniqueness of  solution to BSEE with jump in the Gelfand triple framework. Then we apply the results on BSEE to the specific BSPDIE, and get the well-posedness of both weak solution and strong solution to BSPDIE in Section 3. Finally, in Section 4 we prove the comparison principle for BSPDIE which has a potential application to financial mathematics.

\section{Backward stochastic evolution equation with jump in Hilbert spaces}

Let $(\Omega, \mathscr{F}, P)$ be a complete probability space. Suppose that $\{W_t, 0\leq t\leq T\}$ is a standard one-dimensional Brownian motion and $\{\eta_t, 0\leq t\leq T\}$ is a stationary one-dimensional Poisson point process defined on $(\Omega, \mathscr{F}, P)$, and $W$ and $\eta$ are mutually independent. Define $\mathscr{F}_t$, $0\leq t\leq T$, to be the $\sigma$-algebra generated by $\{W_s, 0\leq s\leq t\}$ and $\{\eta_s, 0\leq s\leq t\}$. Let $\mathbb{F}=\{\mathscr{F}_t, 0\leq t\leq T\}$ be the usual augmentation of filtration, and denote by $\mathscr{P}$ the $\mathbb{F}$-predictable $\sigma$-algebra on $\Omega\times [0,T]$.
The Poisson process $\eta$ induces jumps taking values in a fixed nonempty measurable subset $Z\subset \mathbb{R}$ equipped with a finite characteristic measure $\nu$. The associated counting measure $\mu$ is then defined on the product space $[0,T]\times Z$, and the corresponding compensated Poisson martingale measure is given by
\[ \tilde{\mu}(de,dt):= \mu(de,dt)-\nu(de)dt. \]

We introduce some notation used throughout this paper:
\begin{itemize}
  \item $\mathscr{B}(\Lambda)$: the Borel $\sigma$-algebra on a topological space $\Lambda$;
  \item $X$: a separable Hilbert space equipped with norm $\|\cdot\|_X$;
  \item $S_{\mathbb{F}}^2(0,T; X)$: the space of all $\mathbb{F}$-adapted and c\`adl\`ag processes $f: [0,T] \times \Omega \to X$ such that
  \[
  \|f\|_{S_{\mathbb{F}}^2(0,T; X)} := \left( \mathbb{E} \sup_{0 \leq t \leq T} \|f(t)\|_X^2 \right)^{1/2} < \infty;
  \]
  \item $M_{\mathbb{F}}^2(0,T; X)$: the space of all $\mathbb{F}$-adapted processes $f: [0,T] \times \Omega \to X$ such that
  \[
  \|f\|_{M_{\mathbb{F}}^2(0,T; X)} := \left( \mathbb{E} \int_0^T \|f(t)\|_X^2dt \right)^{1/2} < \infty;
  \]
  \item $M^{\nu,2}(Z; X)$: the space of all $\mathscr{B}(Z)$-measurable functions $f: Z \to X$ such that
  \[
  \|f\|_{M^{\nu,2}(Z; X)} := \left( \int_Z \|f(e)\|_X^2\,\nu(de) \right)^{1/2} < \infty;
  \]
  \item $M_{\mathscr{P}}^{\nu,2}([0,T] \times Z; X)$: the space of all $\mathscr{P} \otimes \mathscr{B}(Z)$-measurable processes $f: [0,T] \times Z \times \Omega \to X$ such that
  \[
  \|f\|_{M_{\mathscr{P}}^{\nu,2}([0,T] \times Z; X)} := \left( \mathbb{E} \int_0^T \int_Z \|f(t,e)\|_X^2\,\nu(de)dt \right)^{1/2} < \infty.
  \]
\end{itemize}

Then we introduce the Gelfand triple $(V,H,V^{*})$, where $V$ and $H$ are separable real Hilbert spaces such that $V$ is continuously and densely embedded into the pivot space $H$, and ${V}^{*}$ represents the dual space of $V$. Identifying $H$ with its dual $H^{*}$ via the Riesz representation theorem, we obtain the following embeddings:
\[
V\subset H=H^{*}\subset V^{*}.
\]

Throughout the paper, we adopt the following notations associated with the Gelfand triple:
\begin{itemize}
  \item $\|\cdot\|_{V}$, $\|\cdot\|_{H}$, and $\|\cdot\|_{V^{*}}$ denote the norms in the spaces $V$, $H$, and $V^{*}$, respectively;
  \item $(\cdot,\cdot)_{H}$ denotes the inner product in $H$;
  \item $\langle\cdot,\cdot\rangle$ denotes the duality pairing between $V^{*}$ and $V$;
  \item $\mathcal{L}(V,V^{*})$ denotes the space of bounded linear operators from $V$ to $V^{*}$.
\end{itemize}

The BSEE with jump we study in this paper has a form
\begin{equation} \label{eq:a2}
\begin{aligned}
u(t) &= \varphi + \int_t^T \big[ A(s)u(s) + f(s, u(s), q(s), r(s,\cdot)) + g(s, u(s), q(s), r(s,\cdot)) \big] \, ds \\
&\quad - \int_t^T q(s)\, dW_s - \int_t^T \int_Z r(s,e) \, \tilde \mu(de, ds), \quad t \in [0,T],
\end{aligned}
\end{equation}
where the coefficients are progressively measurable mappings
\[
\begin{aligned}
&A : \Omega \times [0,T] \to \mathscr{L}(V, V^*), \quad \varphi : \Omega \to H, \\
&f : \Omega \times [0,T] \times V \times H \times M^{\nu,2}(Z; H) \to V^*, \\
&g : \Omega \times [0,T] \times V \times H \times M^{\nu,2}(Z; H) \to H.
\end{aligned}
\]

\begin{defn}\label{def:1.1}
A triple process $(u,q,r)\in S_{\mathbb{F}}^2(0,T;H)\cap M_{\mathbb{F}}^2(0,T;V)\times M_\mathbb{F}^2(0,T;H)\times{M}_\mathscr{P}^{\nu,2}([0,T]\times Z;H)$ is called a solution of BSEE \eqref{eq:a2} if for any $\eta\in V$ and all $t\in[0,T]$, it holds that
\begin{eqnarray*}\label{eq:a3}
(u(t),\eta)_H &=& (\varphi,\eta)_H + \int_t^T \Big[ \langle A(s)u(s),\eta\rangle + \langle f(s,u(s),q(s),r(s,\cdot)),\eta\rangle \nonumber\\
&& + (g(s,u(s),q(s),r(s,\cdot)),\eta)_H \Big]ds - \int_t^T (q(s), \eta)_H dW_s \nonumber\\
&& - \int_t^T \int_Z (r(s,e),\eta)_H \tilde{\mu}(de,ds)\ \ \ \text{a.s.}
\end{eqnarray*}
\end{defn}

The It\^{o } formula for BSEE in the Gelfand triple setting can refer to the forward stochastic equation case given by Gy\"{o}ngy and Krylov \cite{gyo-kry}.
\begin{lem}\label{lem:c1}
Assume that $v,m,v^{*}$ are three processes defined on $\Omega\times[0,T]$ and valued in $V,H,V^{*}$, respectively. For any $\eta\in V$, $v(t), m(t), \langle\eta, v^{*}(t)\rangle$ are $\mathscr{F}_{t}$-measurable for a.e. $t\in[0,T]$, and $m$ is a strongly c\`{a}dl\`{a}g local martingale. If $\varphi\in L_{\mathscr{F}_{T}}^{2}(\Omega;H)$ and for any $\eta\in V$, $(\omega,t)\in\Omega\times[0,T]$, it holds that
\begin{eqnarray*}
    ( \eta,v(t))_H = ( \eta,\varphi)_H
    + \int_{t}^{T} \langle \eta,v^{*}(s) \rangle ds
    + ( \eta, m(T)-m(t) )_H,
\end{eqnarray*}
then there exists a full-measure set $\Omega'\subset \Omega$ and a process $h$ valued in $H$ such that\\
(i) $h(t)$ is $\mathscr{F}_{t}$-measurable for any $t\in[0,T]$ and strongly c\`{a}dl\`{a}g for any $\omega\in\Omega$;\\
(ii) $h(t)=v(t)$ in the space $H$ for a.s. $\omega\in\Omega$ and a.e. $t\in[0,T]$, and $h(T)=\varphi$ for any $\omega\in\Omega'$;\\
(iii) for any $\omega\in \Omega'$ and $t\in[0,T]$,
\begin{eqnarray*}
    \|h(t)\|_{H}^{2} = \|\varphi\|_{H}^{2} + 2\int_{t}^{T}\langle v^{*}(s), v(s)\rangle ds + 2 \int_{t}^{T}( h(s-), d m(s)) - [m ]_{T} + [m]_{t},
\end{eqnarray*}
where $[m]$ denotes the quadratic variation of the $H$-valued martingale.
\end{lem}

To solve BSEE \eqref{eq:a2}, some assumptions are needed.
\begin{ass}\label{ass:1.1}
(A.1) \textbf{(Measurability)}:
The random mapping $A$ is measurable with respect to $\mathscr{P}$, and takes values in $\mathscr{L}(V,V^*)$; the random mappings $f$ and $g$ are measurable with respect to $\mathscr{P} \otimes \mathscr{B}(V \times H \times M^{\nu,2})$, and take values in $V^*$ and $H$, respectively; the random mapping $\varphi$ is measurable with respect to $\mathscr{F}_T$, and takes values in $H$.

(A.2) \textbf{(Integrability)}: The random mappings $(A, f, g, \varphi)$ satisfy
\[
A\in M_{\mathbb{F}}^2(0, T; \mathscr{L}(V, V^*)),
\quad f(\cdot, 0, 0, 0) \in M_{\mathbb{F}}^2(0, T; V^*),
\]
\[
g(\cdot, 0, 0, 0) \in M_{\mathbb{F}}^2(0, T; H),\quad \varphi \in L^2_{\mathscr{F}_T}(\Omega; H).
\]

(A.3) \textbf{(Boundedness)}: The random mapping $A$ is uniformly bounded by a constant $M$, i.e.,
\[
\sup_{(t, \omega) \in [0, T] \times \Omega} \| A(t, \omega) \|_{\mathscr{L}(V, V^*)} \leqslant M.
\]

(A.4) \textbf{(Coercivity)}: For any $\omega \in \Omega$, $t \in [0,T]$ and $u \in V$, there exist constants $\alpha > 0$ and $\lambda \in \mathbb{R}$ such that
\[
\langle A(t)u, u \rangle \leq -\alpha \|u\|_V^2 + \lambda \|u\|_H^2.
\]

(A.5) \textbf{(Lipschitz)}: For any $\omega \in \Omega$, $t \in [0,T]$, $u \in V$ and $(u_1, q_1, r_1), (u_2, q_2, r_2) \in V \times H \times M^{\nu,2}(Z; H)$, there exist constants $l> 0$ and $0 < \kappa < \sqrt{\alpha}$ such that
\[
\| g(t, u_1, q_1, r_1) - g(t, u_2, q_2, r_2) \|_H \leq l \left[ \| u_1 - u_2 \|_V + \| q_1 - q_2 \|_H + \| r_1 - r_2 \|_{M^{\nu,2}(Z; H)} \right],
\]
	\begin{equation*}
		\begin{aligned}	
|\langle f(t, u_1, q_1, r_1) - f(t, u_2, q_2, r_2), u \rangle| \leq &l \| u_1 - u_2 \|_H \| u \|_V + \kappa \| q_1 - q_2 \|_H \| u \|_V\\
&+ l\| q_1 - q_2 \|_H \| u \|_H + \kappa \| r_1 - r_2 \|_{M^{\nu,2}(Z; H)} \| u \|_V.
	\end{aligned}
	\end{equation*}
\end{ass}

The main result of this section is the solvability of BSEE \eqref{eq:a2}.
\begin{thm}\label{thm:1.2}
Assume that Assumption \ref{ass:1.1} holds. Then BSEE \eqref{eq:a2} has a unique solution.
Moreover, let $(A, f, g, \varphi)$ and $(A, \bar f, \bar g, \bar{\varphi})$ be two sets of coefficients satisfying Assumption \ref{ass:1.1}, and $(u, q, r)$ and $(\bar u, \bar q, \bar r)$ be the solutions of BSEE \eqref{eq:a2} with coefficients $(A, f, g, \varphi)$ and $(A, \bar f, \bar g, \bar \varphi)$, respectively, then we have the following estimate
\begin{eqnarray}\label{eq:130}
&& E \sup_{0 \leq t \leq T} \| u(t) - \bar{u}(t) \|_H^2 + E \int_0^T \| u(t) - \bar{u}(t) \|_V^2 \, dt \nonumber \\
&& + E \int_0^T \| q(t) - \bar{q}(t) \|_H^2 \, dt + E \int_0^T \int_Z \| r(t, e) - \bar{r}(t, e) \|_H^2 \nu(de) \, dt \nonumber \\
&\leq& C \left[ E \| \varphi - \bar{\varphi} \|_H^2 + E \int_0^T \| f(t, \bar{u}(t), \bar{q}(t), \bar{r}(t, \cdot)) - \bar{f}(t, \bar{u}(t), \bar{q}(t), \bar{r}(t, \cdot)) \|_{V^*}^2 \, dt \right. \nonumber \\
&& \left. + E \int_0^T \| g(t, \bar{u}(t), \bar{q}(t), \bar{r}(t, \cdot)) - \bar{g}(t, \bar{u}(t), \bar{q}(t), \bar{r}(t, \cdot)) \|_H^2 \, dt \right].
\end{eqnarray}
Here and in the rest of this paper $C$ is a generic constant depending only on given parameters, and we may indicate the parameters it depends on by a bracket. For example, the above $C= C(l, \lambda, \alpha, \kappa, M, T)$ is a constant depending on the parameters $l, \lambda, \alpha, \kappa, M, T$.

In particular, when $(A, \bar f, \bar g, \bar \varphi) = (A, 0, 0, 0)$, it follows from (\ref{eq:130}) that
\begin{eqnarray}\label{eq:131}
&& E \sup_{0 \leq t \leq T} \| u(t) \|_H^2 + E \int_0^T \| u(t) \|_V^2 \, dt + E \int_0^T \| q(t) \|_H^2 \, dt + E \int_0^T \int_Z \| r(t, e) \|_H^2 \nu(de) \, dt \nonumber \\
&\leq& C\left[ E \| \varphi \|_H^2 + E \int_0^T \| f(t, 0, 0, 0) \|_{V^*}^2 \, dt + E \int_0^T \| g(t, 0, 0, 0) \|_H^2 \, dt \right].
\end{eqnarray}
\end{thm}

Before presenting the complete proof of the main theorem, we first investigate the solvability of the following linear BSEE
\begin{equation}\label{eq:1.5}
u(t)=\varphi+\int_t^T [A(s)u(s)+f(s)]ds - \int_t^T q(s)dW_s - \int_t^T\int_Z r(s,e)\,\tilde{\mu}(de,ds),
\end{equation}
where the inhomogeneous term \( f \) is independent of the unknowns \( (u,q,r) \).
\begin{lem} \label{lem:1.3}
Assume that the operator \( A \) satisfies Assumption \ref{ass:1.1}. If $\varphi \in L^2(\Omega,\mathscr{F}_T; H)$ and $f \in M^2_{\mathscr{F}}(0, T; V^*)$, then BSEE \eqref{eq:1.5} admits a unique solution \( (u, q, r) \in M^2_{\mathscr{F}}(0, T; V) \times M^2_{\mathscr{F}}(0, T; H) \times M^2_{\mathscr{F}}(0, T; L^2_\nu(Z; H)) \).
\end{lem}

\begin{proof}
The existence and uniqueness of solutions follow from a standard Galerkin approximation scheme combined with energy estimates and weak convergence techniques, as rigorously established in Lemma 5.4 of \cite{me-do-sh-ta} for BSEEs with jumps.
\end{proof}

We are ready to prove the main Theorem \ref{thm:1.2} in this section.
\begin{proof}[Proof of Theorem \ref{thm:1.2}]
The stability estimate \eqref{eq:130} follows from the standard energy inequality derived via It\^{o}'s formula, using the Lipschitz continuity of \( f \), \( g \) and the coercivity of \( A \). This argument parallels Theorem 5.2 in \cite{me-do-sh-ta}, where similar techniques are applied to a simplified case and minor adaptations suffice for our generalized setting. The specific case \eqref{eq:131} is obtained by substituting \( \bar{f} = \bar{g} = \bar{\varphi} = 0 \) and \( (\bar{u}, \bar{q}, \bar{r}) = (0, 0, 0) \) into \eqref{eq:130}.

To extend the solvability to the nonlinear case, fix \( f_0 \in {M}_{\mathbb{F}}^2(0, T; V^*) \) and consider the parameterized BSEE for \( \rho \in [0,1) \)
\begin{equation}\label{eq:1.7}
\begin{aligned}
u(t) &= \varphi + \int_t^T \Big[A(s)u(s) + \rho f(s, u(s), q(s), r(s,\cdot)) + \rho g(s, u(s), q(s), r(s,\cdot)) + f_0(s)\Big]ds \\
&\quad - \int_t^T q(s)dW_s - \int_t^T \int_Z r(s,e)\,\tilde{\mu}(de,ds), \quad t \in [0,T].
\end{aligned}
\end{equation}
Under Assumption \ref{ass:1.1}, it is clear that the coefficients $(A, \rho f + f_0, \rho g, \varphi)$ satisfy the conditions of Lemma \ref{lem:1.3} for any $\rho \in [0,1)$, due to the boundedness and Lipschitz continuity of $f,g$ and the fact that $\rho < 1$.

We first assume that for some fixed $\rho_0 \in [0,1)$, BSEE \eqref{eq:1.7} admits a unique solution $(u, q, r)$ in the space
\[
\mathcal{H} := S_{\mathbb{F}}^2(0,T;H) \cap M_{\mathbb{F}}^2(0,T;V) \times M_\mathbb{F}^2(0,T;H) \times M_\mathscr{P}^{\nu,2}([0,T]\times Z;H)
\]
%For notational convenience, we define the norm on $\mathcal{H}$ as:
under the norm
\[
\|(u,q,r)\|_{\mathcal{H}}^2 := E \sup_{0 \leq t \leq T} \| u(t) \|_H^2 + E \int_0^T \| u(t) \|_V^2 \, dt + E \int_0^T \| q(t) \|_H^2 \, dt + E \int_0^T \int_Z \| r(t, e) \|_H^2 \nu(de) \, dt.
\]
Then we aim to show that for any $\rho$ sufficiently close to $\rho_0$, BSEE \eqref{eq:1.7} also admits a unique solution in $\mathcal{H}$.

To see this, rewriting \eqref{eq:1.7} in terms of $\rho_0$ yields
\begin{equation*}\label{eq:1.8}
\begin{aligned}
u(t)= &\varphi + \int_t^T \Big[A(s)u(s) + \rho_0 f(s, u(s), q(s), r(s,\cdot)) + \rho_0 g(s, u(s), q(s), r(s,\cdot)) \\
&+ (\rho - \rho_0)f(s, u(s), q(s), r(s,\cdot)) + (\rho - \rho_0)g(s, u(s), q(s), r(s,\cdot)) + f_0(s)\Big]ds \\
&- \int_t^T q(s)dW_s - \int_t^T \int_Z r(s,e)\,\tilde{\mu}(de,ds).
\end{aligned}
\end{equation*}
For any given triplet $(\tilde u, \tilde q, \tilde r) \in \mathcal{H}$, define a mapping $\Phi:(\tilde u, \tilde q, \tilde r)\mapsto(u, q, r)$ by BSEE
\begin{equation*}\label{eq:1.9}
\begin{aligned}
u(t) &= \varphi + \int_t^T \Big[A(s)u(s) + \rho_0 f(s, u(s), q(s), r(s,\cdot)) + \rho_0 g(s, u(s), q(s), r(s,\cdot)) \\
&\qquad + (\rho - \rho_0)f(s, \tilde u(s), \tilde q(s), \tilde r(s,\cdot)) + (\rho - \rho_0)g(s, \tilde u(s), \tilde q(s), \tilde r(s,\cdot)) + f_0(s)\Big]ds \\
&\quad - \int_t^T q(s)dW_s - \int_t^T \int_Z r(s,e)\,\tilde{\mu}(de,ds).
\end{aligned}
\end{equation*}
Bearing in mind that for $\rho = \rho_0$ and any $(\tilde u, \tilde q, \tilde r) \in \mathcal{H}$, BSEE admits a unique solution $(u,q,r) \in \mathcal{H}$, we know that $\Phi$ is a well-defined mapping from $\mathcal{H}$ to itself.

In fact, for $\rho$ close to $\rho_0$, $\Phi$ is a contraction mapping. We take $(\tilde u_i, \tilde q_i, \tilde r_i) \in \mathcal{H}$, $i=1,2$, and denote $(u_i, q_i, r_i) = \Phi(\tilde u_i, \tilde q_i, \tilde r_i)$.
Using a priori estimate \eqref{eq:130}, together with the Lipschitz continuity of $f,g$, we have
\begin{align*}
\|(u_1-u_2,q_1-q_2,r_1-r_2)\|_{\mathcal{H}}^2 \leq C |\rho - \rho_0|^2 \left( \|(\tilde u_1, \tilde q_1, \tilde r_1) - (\tilde u_2, \tilde q_2, \tilde r_2)\|_{\mathcal{H}}^2 \right),
\end{align*}
where $C$ is a constant depending only on $(l, \lambda, \alpha, \kappa, M, T)$ and independent of $\rho$ and $\rho_0$.
Choose $\delta := \frac{1}{2\sqrt{C}}$. Then, for any $\rho$ satisfying $|\rho - \rho_0| < \delta$, the mapping $\Phi$ is a contraction in the complete metric space $\mathcal{H}$. % with contraction constant $\frac{1}{2}$.
By Banach's fixed-point theorem, BSEE \eqref{eq:1.7} admits a unique solution in this case.
Note that since the constant $C$ %depends only on the fixed parameters $(C, \lambda, \alpha, \kappa, M, T)$ through the Lipschitz and coercivity conditions, and
is independent of $\rho_0$, the step size $\delta$ stays unchanged throughout the continuation procedure.

Finally, since BSEE \eqref{eq:1.7} is solvable at $\rho = 0$ due to Lemma \ref{lem:1.3}, we construct a finite sequence $0 = \rho_0 < \rho_1 < \cdots < \rho_n = 1$ such that $|\rho_{i+1} - \rho_i| < \delta$ for all $i=0,1,\cdots,n-1$, and successively solve \eqref{eq:1.7} for each $\rho_i$ based on the contraction principle. This allows us to reach $\rho = 1$ in finite steps, and in this way the existence of solution to BSEE \eqref{eq:a2} follows.

The uniqueness of solution to BSEE \eqref{eq:a2} is an immediate result from a priori estimate \eqref{eq:130}. Indeed, if $(u_1, q_1, r_1)$ and $(u_2, q_2, r_2)$ are two solutions of \eqref{eq:a2}, then applying \eqref{eq:130} with $\bar{f} = f$, $\bar{g} = g$, $\bar{\varphi} = \varphi$ yields
\[
\|(u_1-u_2,q_1-q_2,r_1-r_2)\|_{\mathcal{H}}^2 \leq 0,
\]
which implies that the two solutions coincide with each other.
\end{proof}

\section{Backward Stochastic Partial Differential Integral Equation}

In this section, we study a type of BSPDIE which can be regarded as a concrete realization of BSEE \eqref{eq:a2} by specifying an appropriate form of the operator \( A \).
For notational simplicity, define
$$D_i=\partial_{x^i}\ \ \ \text{and}\ \ \ D_{ij}=\partial^2_{x^ix^j},\ \ \ i,j=1,2,\cdots,d,$$
and for
the differential operator associated with a multi-index \( \alpha = (\alpha_1, \dots, \alpha_d) \in \mathbb{N}^d \) with \( |\alpha| = \alpha_1 + \cdots + \alpha_d \), define
\[
D^\alpha = \left( \partial_{x^1} \right)^{\alpha_1} \left( \partial_{x^2} \right)^{\alpha_2} \cdots \left( \partial_{x^d} \right)^{\alpha_d}.
\]
We also denote by \( Du \) and \( D^2u \) the gradient vector and the Hessian matrix of a function \( u \) defined on \( \mathbb{R}^d \), respectively.

Throughout this paper, when a vector-valued or matrix-valued function is said to belong to a given function space (e.g., \( Du \in L^2(\mathbb{R}^d) \)), it is understood that all of its components lie in that space. We also adopt the Einstein summation convention to simplify notation in expressions involving indexed components.

%We study Sobolev solutions to BSPDIEs.
Denote by \( W^{m,p}(\mathbb{R}^d) \) the standard Sobolev space consisting of functions defined on $\mathbb{R}^d$ whose weak derivatives up to order \( m \in \mathbb{Z}^+ \) belong to \( L^p(\mathbb{R}^d) \), for \( 1 \leq p < \infty \). In particular, we write \( H^m := W^{m,2}(\mathbb{R}^d) \) for \( m \in \mathbb{N} \), and identify \( H^0 \) with \( L^2(\mathbb{R}^d) \). Each \( H^m \) space is a Hilbert space with its dual denoted by \( H^{-m} \).

We adopt the following simple notations for spaces:
\[
\begin{aligned}
& H^0 = L^2 = H^0(\mathbb{R}^d) = L^2(\mathbb{R}^d),\ \ \ \ \ \mathbb{H}^n = \mathbb{H}^n(\mathbb{R}^d) := M_{\mathbb{F}}^2(0,T; H^n), \\
& \mathbb{H}^{\nu, n} = \mathbb{H}^{\nu, n}(\mathbb{R}^d) := M_\mathscr{P}^{\nu,2}([0,T] \times Z; H^n),\ \ \ \ \ \mathbb{S}^2({H}^n) := S_{\mathbb{F}}^2(0,T;{H}^n),
\end{aligned}
\]
and the following simple notations for inner products and norms:
\[
\begin{aligned}
( \cdot,\cdot )_n := &( \cdot,\cdot)_{H^n},\ \ \ \ \ \| \cdot \|_n := \| \cdot \|_{H^n}, \ \ \ \ \ \| \cdot \|_{\nu,n} := \| \cdot \|_{M^{\nu,2}(Z;H^n)},\\
&|||\cdot |||_n:=\|\cdot\|_{\mathbb{H}^n},\ \ \ \ \ |||\cdot|||_{\nu,n} :=\|\cdot\|_{\mathbb{H}^{\nu, n}}.
\end{aligned}
\]
For $(t,x)\in[0,T]\times \mathbb R^d$, the BSPDIE we are concerned with has a divergence form
\begin{eqnarray}\label{eq:2.1}
u(t,x)&=&\varphi(x)+\int_t^T\bigg(
      D_{i}(a^{ij}(s,x)D_{j}u(s,x)) +b^i(s,x)D_{i}u(s,x)+c(s,x)u(s,x)\nonumber\\
      &&\ \ \ \ \ \ \ \ \ \ \ \ \ \ \ \ +\int_{Z}\Big[u(s,x+k(s,e,x))-u(s,x)-k^i(s,e,x)D_{i}u(s,x)\Big]\nu (de)\nonumber\\
      &&\ \ \ \ \ \ \ \ \ \ \ \ \ \ \ \  +D_{i}(\sigma^i(s,x)q(s,x)) +v(s,x)q(s,x)+f_0(s,x)\nonumber\\
      &&\ \ \ \ \ \ \ \ \ \ \ \ \ \ \ \  +\int_{Z}\Big[r(s,e,x+k(s,e,x))-r(s,e,x)\Big]\nu(de)\bigg)ds\nonumber\\
      &&-\int_t^Tq(s,x)dW_s-\int_t^T\int_{Z}r(s,e,x)\tilde \mu(de,ds),
\end{eqnarray}
or equivalently, a non-divergence form
\begin{equation} \label{eq:3.2}
\begin{aligned}
u(t,x) = &\varphi(x) + \int_t^T \bigg(
a^{ij}(s,x)D_{ij}u(s,x) + b^i(s,x)D_{i}u(s,x) + c(s,x)u(s,x) \\
&\ \ \ \ \ \ \ \ \ \ \ \ \ \ \ \ + \int_Z \left[ u(s,x+k(s,e,x)) - u(s,x) - k^i(s,e,x)D_{i}u(s,x) \right] \nu(de) \\
&\ \ \ \ \ \ \ \ \ \ \ \ \ \ \ \ + \sigma^i(s,x)D_{i}q(s,x) + v(s,x)q(s,x) + f_0(s,x) \\
&\ \ \ \ \ \ \ \ \ \ \ \ \ \ \ \ + \int_Z \left[ r(s,e,x+k(s,e,x)) - r(s,e,x) \right] \nu(de) \bigg) ds \\
& - \int_t^T q(s,x)dW_s - \int_t^T \int_Z r(s,e,x)\tilde{\mu}(de,ds).
\end{aligned}
\end{equation}
Here all coefficients $a=[a^{ij}]:\Omega\times[0,T]\times\mathbb{R}^d\to\mathbb{R}^d\times\mathbb{R}^d,\ b=(b^1,\cdots,b^d)^*:\Omega\times[0,T]\times\mathbb{R}^d\to\mathbb{R}^d,\ c:\Omega\times[0,T]\times\mathbb{R}^d\to\mathbb{R}^1,\ k=(k^1,\cdots,k^d)^*:\Omega\times[0,T]\times Z\times\mathbb{R}^d\to\mathbb{R}^d,\ \sigma=(\sigma^1,\cdots,\sigma^d)^*:\Omega\times[0,T]\times\mathbb{R}^d\to\mathbb{R}^d,\ v:\Omega\times[0,T]\times\mathbb{R}^d\to\mathbb{R}^1$, the non-homogeneous term $f_0:\Omega\times[0,T]\times\mathbb{R}^d\to\mathbb{R}^1$, and the terminal value $\varphi:\Omega\times\mathbb{R}^d\to\mathbb{R}^1$ could be random.

Then we give the definitions for weak solution and strong solution to the concerned BSPDIEs.
\begin{defn}\label{def:3.1}
A triple of processes \((u, q, r) \in \mathbb{S}^2(H^0) \cap \mathbb{H}^1 \times \mathbb{H}^0 \times \mathbb{H}^{\nu,0}\) is called a weak solution to BSPDIE \eqref{eq:2.1} if, for any \(t \in [0, T]\) and test function \(\eta \in C_0^\infty(\mathbb{R}^d)\),
\begin{eqnarray*}\label{eq:3.222}
&&\int_{\mathbb{R}^d} u(t, x) \eta(x) \, dx \nonumber\\
&=& \int_{\mathbb{R}^d} \varphi(x) \eta(x) \, dx + \int_t^T \int_{\mathbb{R}^d} \Bigg(
- a^{ij}(s, x) D_{j} u(s, x) D_{i} \eta(x) + b^i(s, x) D_{i} u(s, x) \eta(x)  \nonumber
\\&&+ c(s, x) u(s, x) \eta(x) \nonumber + \int_Z \Big[ u(s, x + k(s, e, x)) - u(s, x) - k^i(s, e, x) D_{i} u(s, x) \Big] \eta(x) \nu(de) \nonumber\\
&&- \sigma^i(s, x) q(s, x) D_{i} \eta(x) + v(s, x) q(s, x) \eta(x) + f_0(s, x) \eta(x) \nonumber\\
&&+ \int_Z \Big[ r(s, x + k(s, e, x)) - r(s, e, x) \Big] \eta(x) \nu(de)
\Bigg)dxds \nonumber\\
&&- \int_t^T \int_{\mathbb{R}^d} q(s, x) \eta(x)dxdW_s - \int_t^T \int_Z \int_{\mathbb{R}^d} r(s, e, x) \eta(x)dx\tilde{\mu}(de, ds).
\end{eqnarray*}
\end{defn}

\begin{defn}\label{def:3.2}
A triple of processes \((u, q, r) \in \mathbb{S}^2(H^1) \cap \mathbb{H}^2 \times \mathbb{H}^1 \times \mathbb{H}^{\nu,1}\) is called a strong solution to BSPDIE \eqref{eq:3.2} if, for any \(t \in [0,T]\) and a.e. \(x \in \mathbb{R}^d\),
\begin{eqnarray*}\label{eq:3.3}
u(t,x) &=& \varphi(x) + \int_t^T \Bigg(
a^{ij}(s,x) D_{ij} u(s,x) + b^i(s,x) D_{i} u(s,x) + c(s,x) u(s,x) \nonumber\\
&&+ \int_Z \Big[ u(s, x + k(s,e,x)) - u(s,x) - k^i(s,e,x) D_{i} u(s,x) \Big] \nu(de) \nonumber\\
&&+ \sigma^i(s,x) D_{i} q(s,x) + v(s,x) q(s,x) + f_0(s,x) \nonumber\\
&&+ \int_Z \Big[ r(s,e, x + k(s,e,x)) - r(s,e,x) \Big] \nu(de)
\Bigg) ds \nonumber\\
&&- \int_t^T q(s,x)dW_s - \int_t^T \int_Z r(s,e,x)\tilde{\mu}(de,ds).
\end{eqnarray*}
\end{defn}

\begin{ass}\label{ass:b1}
The functions \( a, b, c, \sigma, v \), and \( f_0 \) are measurable with respect to \( \mathscr{P} \times \mathscr{B}(\mathbb{R}^d) \), and take values in the spaces of real symmetric \( d \times d \) matrices, \( \mathbb{R}^d \), \( \mathbb{R} \), \( \mathbb{R}^{d} \), \( \mathbb{R} \), and \( \mathbb{R} \), respectively.
The function \( k \) is measurable with respect to \( \mathscr{P} \times \mathscr{B}(Z) \times \mathscr{B}(\mathbb{R}^d) \), and takes values in \( \mathbb{R}^d \).
The terminal condition \( \varphi \) is measurable with respect to \( \mathscr{F}_T \times \mathscr{B}(\mathbb{R}^d) \), and takes values in \( \mathbb{R} \).
\end{ass}

\begin{ass}\label{ass:b2}
\textbf{(Uniform Parabolicity)}
There exist constants \( \alpha > 0 \), \( \delta > 0 \), and \( 0 < \kappa < \sqrt{\alpha} \) such that for any \( (\omega, t, e, x) \in \Omega \times [0,T] \times Z \times \mathbb{R}^d \), the following (i)--(iii) hold:
\begin{itemize}
  \item[(i)] uniform ellipticity: for all \( \xi \in \mathbb{R}^d \),
        \[
        a^{ij}(t,x) \xi_i \xi_j \geq \alpha |\xi|^2;
        \]
  \item[(ii)] diffusion coefficient bound:
        \[
        |\sigma(t,x)|^2 = \sum_{i=1}^d |\sigma^i(t,x)|^2 \leq \kappa^2;
        \]
  \item[(iii)] non-degeneracy of jump transformation:
        \[
        \left|\det\left(I + \partial_x k(t,e,x)\right)\right| \geq \delta.
        \]
\end{itemize}
\end{ass}
\begin{ass}\label{mz4}
The coefficients \( {{a}}, b, c,{{\sigma}}, v, k \) are uniformly bounded by a constant \( L > 0 \).
\end{ass}

With Assumptions \ref{ass:b1}--\ref{mz4}, together with the integrable conditions on the generator $(\varphi,f_0)$, we first establish the existence and uniqueness of weak solution to the concerned BSPDIE.
\begin{thm}\label{thm:3.1}
Assume that Assumptions \ref{ass:b1}--\ref{mz4} hold. If
$f_0 \in \mathbb{H}^{-1}$ and $\varphi \in L^2_{\mathscr{F}_T}(\Omega; L^2)$,
then BSPDIE \eqref{eq:2.1} admits a unique weak solution $(u,q,r) \in \mathbb{S}^2(H^0) \cap \mathbb{H}^1 \times \mathbb{H}^0 \times \mathbb{H}^{\nu,0}$ satisfying
\[
\mathbb{E} \left[\sup_{0 \leq t \leq T} \|u(t)\|_0^2 \right] + |||u|||_1^2 + |||q|||_0^2 + |||r|||_{\nu,0}^2 \leq C \left( |||f_0|||_{-1}^2 + \mathbb{E}[\|\varphi\|_0^2] \right),
\]
where $C = C(L, \kappa, \alpha, \delta, T)$ is a positive constant.
\end{thm}

\begin{proof} We utilize the Gelfand triple framework and take \((V, H, V^*)
= \big(H^1, H^0, H^{-1})\). For any $t\in[0,T]$ and \((u, q, r) \in H^1 \times H^0 \times M^{\nu,2}(Z; H^0)\), we define some operators:
\[\begin{aligned} &A(t)u:= D_{i}\big(a^{ij}(t)D_{j}u\big) + b^i(t)D_{i}u + c(t)u  + \int_Z \big[u(\cdot + k(t,e,\cdot)) - u - k^i(t,e,\cdot)D_{i}u\big] \nu(de), \\ &f(t,u,q,r):= D_{i}\big(\sigma^i(t)q\big) + v(t)q + f_0(t), \\ &g(t,u,q,r):= \int_Z \big[r(e,\cdot + k(t,e,\cdot)) - r(e,\cdot)\big] \nu(de). \end{aligned}\]
Actually, BSPDIE \eqref{eq:2.1} is a specific type of BSEE \eqref{eq:a2} with above \((A, f, g)\). Next we verify that all conditions in Assumption \ref{ass:1.1} for \((A, f, g)\) in BSPDIE \eqref{eq:2.1} are satisfied.

   \textbf{Step 1: Verification of Conditions (A.1)--(A.3)}

    The measurable condition (A.1), integrable condition (A.2), and uniform bounded condition (A.3) for the random mapping $A$ can be directly verified by
    the conditions in  Theorem \ref{thm:3.1}.

\textbf{Step 2: Coercivity of \(A\) (Condition A.4)} For any \(\eta \in C_0^\infty(\mathbb{R}^d)\), we divide the estimate of the bilinear form \(\langle A(t)\eta, \eta \rangle\) for $t \in [0,T]$ into three parts.

 \textbf{a. Principal term:} by the uniform ellipticity of \(a^{ij}\) and integration by parts, we have \[ \int_{\mathbb{R}^d} -a^{ij}(t,x) D_{i}\eta D_{j}\eta \, dx \leq -\alpha \|\nabla \eta\|_{L^2}^2 = -\alpha \|\eta\|_{H^1}^2 + \alpha \|\eta\|_{L^2}^2. \]

 \textbf{b. Lower-order terms:} by the boundedness of \(b^i, c\) and Young's inequality, for arbitrary $\varepsilon>0$, we have \[ \int_{\mathbb{R}^d} \big(b^i D_{i}\eta \cdot \eta + c \eta^2\big) dx \leq \varepsilon \|\eta\|_{H^1}^2 + C(\varepsilon, L) \|\eta\|_{L^2}^2. \]

 \textbf{c. Jump term:}
 by the absolute value inequality, %\(\vert a - b - c\vert\leq\vert a\vert+\vert b\vert+\vert c\vert\),
we have
\begin{align*}
&\int_{\mathbb{R}^{d}}\int_{Z}[ \eta(x + k(t,e,x)) - \eta(x) - k^{i}(t,e,x) D_{i}\eta(x) ] \eta(x) \nu (de) dx\\
\leq& \int_{\mathbb{R}^{d}}\int_{Z}\eta(x + k(t,e,x))\nu(de)\eta(x) dx + \int_{Z}\int_{\mathbb{R}^{d}}\vert\eta(x)\vert ^{2}dx\nu (de) \\
&+\int_{Z}\int_{\mathbb{R}^{d}}\vert k^{i}(t,e,x) D_{i}\eta(x) \eta(x)\vert dx \nu (de).
\end{align*}
For arbitrary $\varepsilon>0$, it follows from H\"{o}lder's inequality and Young's inequality that
\begin{align*}
&\int_{\mathbb{R}^{d}}\int_{Z}\eta(x + k(t,e,x))\nu(de)\eta(x) dx
+ \int_{Z}\int_{\mathbb{R}^{d}}|\eta(x)|^{2}dx\nu (de) \\
&+\int_{Z}\int_{\mathbb{R}^{d}}|k^{i}(t,e,x) D_{i}\eta(x) \eta(x)| dx \nu (de) \\
\leq &\frac{1}{2}\int_{\mathbb{R}^{d}}\left|\int_{Z}\eta(x + k(t,e,x)) \nu(de)\right|^{2}dx
+ \varepsilon\|\eta\|_1^2 + C(L,\varepsilon)\|\eta\|_0^2.
\end{align*}
Based on the property of the measure \(\nu\) on \(Z\), we have
\begin{align*}
&\frac{1}{2}\int_{\mathbb{R}^{d}}\vert \int_{Z}\eta(x + k(t,e,x)) \nu(de)\vert ^{2}dx+\varepsilon\|\eta\|_1^2+C(L,\varepsilon)\|\eta\|_0^2\\
\leq& \frac{1}{2}\nu(Z)\int_{Z}\int_{\mathbb{R}^{d}}\vert\eta(x + k(t,e,x))\vert ^{2}dx\nu(de)+\varepsilon\|\eta\|_1^2+C(L,\varepsilon)\|\eta\|_0^2.
\end{align*}
Set \(y = x + k(t,e,x)\), and we do the change of variables. Noticing the Jacobian determinant \(\left|\det\left(I + \partial_x k(t,e,x)\right)\right| \geq \delta>0\), we have
\begin{align*}
&\frac{1}{2}\nu(Z)\int_{Z}\int_{\mathbb{R}^{d}}\vert\eta(x + k(t,e,x))\vert ^{2}dx\nu(de)+\varepsilon\|\eta\|_1^2+C(L,\varepsilon)\|\eta\|_0^2\\
=&\frac{1}{2}\nu(Z)\int_{Z}\int_{\mathbb{R}^{d}}\vert\eta(y)\vert ^{2}\vert\det (I + \partial_{x}k(t,e,x))\vert^{-1}dy\nu(de)+\varepsilon\|\eta\|_1^2+C(L,\delta,\varepsilon)\|\eta\|_0^2\\
\leq&\frac{1}{2}\frac{\nu(Z) ^{2}}{\delta }\|\eta\|_0^2+\varepsilon\|\eta\|_1^2+C(L,\delta,\varepsilon)\|\eta\|_0^2,
\end{align*}
which implies
\begin{align*}
&\int_{\mathbb{R}^{d}}\int_{Z}\big[ \eta(x + k(t,e,x)) - \eta(x) - k^{i}(t,e,x) D_{i}\eta(x) \big] \eta(x) \nu(de) \, dx \notag \\
\leq\
&\frac{1}{2}\frac{\nu(Z)^2}{\delta} \|\eta\|_0^2 + \varepsilon \|\eta\|_1^2 + C(L,\delta,\varepsilon)\|\eta\|_0^2.
\label{eq:3.131}
\end{align*}

In conclusion, choosing $\varepsilon = {\alpha\over4}$ and combining estimates in above $\textbf{a,b,c}$, we have the desired coercivity
$$
\langle A(t)\eta, \eta \rangle \leq -\frac{\alpha}{2} \|\eta\|_1^2 + C \|\eta\|_0^2.
$$

 \textbf{Step 3: Lipschitz Continuity of \(f\) and \(g\) (Condition A.5)}

Take any \((u_1, q_1, r_1), (u_2, q_2, r_2) \in H^1 \times H^0 \times M^{\nu,2}(Z; H^0)\), and we verify the Lipschitz continuity of $f$ and $g$ in turn. For $f$, due to the bound $(\sigma^{i}(t,x))(\sigma^{i}(t,x))^{*} \leq \kappa^2 I$,
    for any \(u \in H^1\), it turns out that
\[
\begin{aligned}
&|\langle f(t,u_1,q_1,r_1) - f(t,u_2,q_2,r_2), u \rangle| \\
=& |\langle D_{i}(\sigma^i(t)(q_1 - q_2)) + v(t)(q_1 - q_2), u \rangle| \\
=& \left| -\int_{\mathbb{R}^d} \sigma^i(t)(q_1 - q_2)(x) \cdot D_{i} u(x) \, dx + \int_{\mathbb{R}^d} v(t)(q_1 - q_2)(x) \cdot u(x) \, dx \right| \\
\leq& \kappa \|q_1 - q_2\|_H \|u\|_V + C \|q_1 - q_2\|_H \|u\|_H.
\end{aligned}
\]
For \(g\), the Lipschitz continuity follows immediately from
   \begin{eqnarray*}\label{eq:3.15}
&&\Vert g( t,u_1,q_1,r_1) -g(t,u_2,q_2,r_2)
\Vert _0\nonumber\\
&=&\|\int_{Z}\bigg(\Big(r_1\big(t, e,\cdot+k(t,e,\cdot)
\big)-r_2\big(t, e,\cdot+k(t,e,\cdot)
\big)\Big) -\big(r_1(t, e,\cdot)-r_2(t, e,\cdot)\big) \bigg)\nu( de)\|_0\nonumber\\
&\leq&\|\int_{Z}\big(r_1(t, e,\cdot)-r_2(t, e,\cdot\big)|\det (I+\partial_{x}k(t,e,x)|^{-1}\nu( de)\|_0\nonumber\\
&&+\|\int_{Z}\big(r_1(t, e,\cdot)-r_2(t, e,\cdot)\big)\nu( de)\|_0\nonumber\\
&=&C(\delta,\nu(Z) ) \Vert r_1(t,\cdot,\cdot)-r_2(t,\cdot,\cdot)%
\Vert_ {\nu,0}.
\end{eqnarray*}

Now we have verified all conditions in Assumption \ref {ass:1.1}
for the operators $(A,f,g)$. By Theorem \ref{thm:1.2}, there exists a unique (weak) solution $(u, q, r) \in \mathbb{S}^2(L^2) \cap \mathbb{H}^1 \times \mathbb{H}^0 \times \mathbb{H}^{\nu,0}$ to BSPDIE \eqref{eq:2.1} satisfying the energy estimate
$$
\mathbb{E}\left[\sup_{0 \leq t \leq T} \|u(t)\|_0^2\right] + |||u|||_1^2 + |||q|||_0^2 + |||r|||_{\nu,0}^2 \leq C \left(|||f_0|||_{-1}^2 + \mathbb{E}[\|\varphi\|_0^2]\right),
$$
where $C = C(L, \kappa, \alpha, \delta, T)$. This completes the proof of Theorem \ref{thm:3.1}.
\end{proof}

To investigate the strong solution to BSPDIE \eqref{eq:3.2}, we need an additional condition.
\begin{ass}\label{ass:b3}
There exists a continuous and increasing function \( \gamma: \mathbb{R}^+ \to \mathbb{R}^+ \) with \( \gamma(0) = 0 \) such that for any \( \omega \in \Omega \), \( t \in [0,T] \) and \( x,y \in \mathbb{R}^d \),
\[
|a(t,x) - a(t,y)| + |\sigma(t,x) - \sigma(t,y)| \leq \gamma(|x - y|).
\]
\end{ass}
In the rest of this section, we prove the main theorem related to strong solution.
\begin{thm}\label{thm:b2}
Assume that Assumptions \ref{ass:b1}--\ref{ass:b3} hold. If $f_0 \in \mathbb{H}^0$ and $\varphi \in L^2_{\mathscr{F}_T}(\Omega; H^1)$, then
BSPDIE \eqref{eq:3.2} admits a unique strong solution $(u,q,r)\in \mathbb{S}^2(H^1) \cap \mathbb{H}^2, \times \mathbb{H}^1\times \mathbb{H}^{\nu,1}$ satisfying
\begin{equation}\label{eq:3.18}
E\left[ \sup_{0 \leq t \leq T} \| u(t) \|_1^2 \right] + ||| u |||_2^2 + ||| q |||_1^2 + ||| r |||_{\nu,1}^2 \leq C \left( ||| f_0 |||_0^2 + E\| \varphi \|_1^2 \right),
\end{equation}
where \( {{C = C(L,\kappa,\alpha,\delta,T)}} \) is a positive constant.
\end{thm}

To start the proof of Theorem \ref{thm:b2},  we study a simple version of BSPDIE \eqref{eq:3.2} whose coefficients $a$ and $\sigma$ are independent of the variable $x$.
\begin{prop}\label{prop:c1}
Assume that Assumptions \ref{ass:b1}--\ref{mz4} hold and the coefficients \( a, \sigma \) are independent of \( x \).
If $f_0 \in \mathbb{H}^0$ and $\varphi \in L^2_{\mathscr{F}_T}(\Omega; H^1)$, then
BSPDIE \eqref{eq:3.2} admits a unique strong solution $(u,q,r)\in \mathbb{S}^2(H^1) \cap \mathbb{H}^2, \times \mathbb{H}^1\times \mathbb{H}^{\nu,1}$,
and the estimate \eqref{eq:3.18} holds.
\end{prop}
\begin{proof}
The proof is divided into two steps.

\textbf{Step 1.} In this step, we consider an even more simple version of BSPDIE \eqref{eq:3.2} with all lower-order coefficients vanish, i.e.,
\[
b = 0, \quad c = 0, \quad v = 0, \quad k = 0.
\]

In this case, BSPDIE \eqref{eq:3.2} reduces to
\begin{equation}\label{eq:3.7}
\begin{aligned}
u(t,x) = &\varphi(x) + \int_t^T \bigg(
a^{ij}(s)D_{ij} u(s,x) + \sigma^i(t) D_{i} q(s,x) + f_0(s,x) \bigg) ds \\
& - \int_t^T q(s,x)dW_s - \int_t^T \int_Z r(s,e,x)\tilde{\mu}(de,ds).
\end{aligned}
\end{equation}
We apply Theorem~\ref{thm:1.2} with the Gelfand triple $(V, H, V^*)
=(H^2, H^1, L^2)$.
%where the identification \( V^* = H^0 = L^2 \) ensures that all pairings are properly defined within %standard Sobolev spaces.

This setting is well suited for second-order equations. The operator \(A(t)\), $t\in[0,T]$, is a mapping from \(H^2\) to \(L^2\), and the forcing term \(f_0(t)\), $t\in[0,T]$, lies in \(L^2\), which allows us to define the variational formulation in the dual space \(V^*=L^2\). Moreover, the energy estimates can be carried out within the $H^1$ framework, where the main solution resides.

For any $t\in[0,T]$ and \((u, q, r) \in H^2 \times H^1 \times M^{\nu,2}(Z; H^1)\), we define related operators:
\[
A(t)u := a^{ij}(t) D_{ij} u, \quad f(t,u,q,r) := \sigma^i(t) D_{i} q + f_0(t), \quad g(t,u,q,r) := 0.
\]
Then for any $\eta\in H^2$, the dual pairings of operators are
\[
\langle A(t)u, \eta \rangle = -\left( a^{ij}(t) D_{i} u, D_{j} \eta \right)_1\ \ \ \text{and}\ \ \
\langle f(t,u,q,r), \eta \rangle = -\left( \sigma^i(t) q, D_{i} \eta \right)_1 + \left( f_0(t), (I - \Delta)\eta \right)_0.
\]
Here the use of the operator \(I - \Delta\) ensures that the \(L^2\)-function \(f_0(t)\) can be paired with the second-order test function $\eta$.

We verify that \(A\) satisfies the coercivity condition (A.4) in Assumption \ref{ass:1.1}. By integration by parts and the uniform ellipticity of \(a^{ij}\), we have
\begin{align*}
\langle A(t)\eta, \eta \rangle &= -\int a^{ij}(t) \, D_{i} \eta(x) \, D_{j} \eta(x) \, dx - \int a^{ij}(t) \, D_{i\ell} \eta(x) \, D_{j\ell} \eta(x) \, dx \\
&\leq -\alpha \| D \eta \|_0^2 - \alpha \sum_{\ell=1}^d \| D(D_{\ell} \eta) \|_0^2 = -\alpha \left( \| \eta \|_1^2 + \sum_{\ell=1}^d \| D(D_{\ell} \eta) \|_0^2 \right) = -\alpha \| \eta \|_2^2 + \alpha \| \eta \|_1^2.
\end{align*}

Let's turn to the Lipschitz continuity condition of non-homogeneous term \( f\).
For any \( q_1, q_2 \in H^1 \) and \( \eta \in H^2 \),
since \( \sigma^i \) is independent of the spatial variable \( x \), we have
\begin{align*}
\langle f(t, u, q_1, r) - f(t, u, q_2, r), \eta \rangle &= - \left( \sigma^i(t)(q_1 - q_2), D_{i} \eta \right)_1\\
&= - \int \sigma^i(t)(q_1 - q_2) D_{i} \eta \, dx - \int D[\sigma^i(t)(q_1 - q_2)] \cdot D (D_{i} \eta)dx\\
&= - \sigma^i(t) \left[ \int (q_1 - q_2) D_{i} \eta \, dx + \int D(q_1 - q_2) \cdot D (D_{i} \eta)dx \right]\\
%Next, applying the condition \( (\sigma^{i}(t,x))(\sigma^{i}(t,x))^{*} \leq \kappa^2 I \) and using the %Cauchy-Schwarz inequality, we obtain:
&\leq \kappa \left( \| q_1 - q_2 \|_0 \cdot \| D \eta\|_0 + \| D (q_1 - q_2) \|_0 \cdot \| D^2 \eta \|_0 \right).
\end{align*}
Notice
\[
\| q_1 - q_2 \|_0 \leq \| q_1 - q_2 \|_1 \quad \text{and} \quad \| D \eta \|_0 \leq \| \eta \|_2.
\]
We conclude
\[
|\langle f(t, u, q_1, r) - f(t, u, q_2, r), u \rangle| \leq 2\kappa \| q_1 - q_2 \|_1 \cdot \| \eta \|_2,
\]
which implies the Lipschitz continuity of \( f(t, u, q, r) \) with respect to \( q \).

By Theorem \ref{thm:1.2}, we know that BSPDIE \eqref{eq:3.7} admits a unique solution \((u, q, r) \in \mathbb{S}^2(H^1) \cap \mathbb{H}^2 \times \mathbb{H}^1 \times \mathbb{H}^{\nu,1}\) and a priori estimate \eqref{eq:3.18} holds true. In particular, for any test function \(\eta \in C_0^\infty(\mathbb{R}^d)\), the triple \((u,q,r)\) satisfies the variational identity
\[
\begin{aligned}
(u(t), \eta)_1 &= (\varphi, \eta)_1 + \int_t^T \Big[
- \big( a^{ij}(s) D_{i} u(s), D_{j} \eta \big)_1
- \big( \sigma^i(s) q(s), D_{i} \eta \big)_1
+ \big( f_0(s), (I - \Delta)\eta \big)_0 \Big] ds \\
&\quad - \int_t^T (q(s), \eta)_1 \, dW_s
- \int_t^T \int_Z (r(s,e), \eta)_1 \, \tilde{\mu}(de,ds).
\end{aligned}
\]
For the convenience of subsequent arguments, we take the test function in the form \(\eta = (I - \Delta)^{-1} \eta\) instead. Then the variational identity becomes
\[
\begin{aligned}
(u(t), (I - \Delta)^{-1} \eta)_1
=& (\varphi, (I - \Delta)^{-1} \eta)_1 + \int_t^T \Big[
- \big( a^{ij}(s) D_{i} u(s), (I - \Delta)^{-1} D_{j} \eta \big)_1 \\
&- \big( \sigma^i(s) q(s), (I - \Delta)^{-1} D_{i} \eta \big)_1
+ \big( f_0(s), \eta \big)_0 \Big] ds \\
&- \int_t^T (q(s), (I - \Delta)^{-1} \eta)_1 \, dW_s
- \int_t^T \int_Z (r(s,e), (I - \Delta)^{-1} \eta)_1 \, \tilde{\mu}(de,ds).
\end{aligned}
\]
By Green's formula, the above variational form can be equivalently rewritten in the standard \(L^2\) framework, i.e., for any \(t \in [0,T]\) and a.e. \(x \in \mathbb{R}^d\),
\[
\begin{aligned}
(u(t), \eta)_0 &= (\varphi, \eta)_0 + \int_t^T \Big[
- \big( a^{ij}(s) D_{ij} u(s), \eta \big)_0
- \big( \sigma^i(s) D_{i} q(s), \eta \big)_0
+ \big( f_0(s), \eta \big)_0 \Big] ds \\
&\quad - \int_t^T (q(s), \eta)_0 \, dW_s
- \int_t^T \int_Z (r(s,e), \eta)_0 \, \tilde{\mu}(de,ds).
\end{aligned}
\]
This confirms that the triple \((u, q, r)\) is a strong solution to BSPDIE \eqref{eq:3.7}.

Moreover, according to Definitions \ref{def:3.1} and \ref{def:3.2}, a strong solution is also a weak solution, which together with the uniqueness of weak solution in Theorem \ref{thm:1.2} leads to the uniqueness of strong solution to BSPDIE \eqref{eq:3.7}.

%Hence, the triple \((u, q, r)\) is the unique strong solution of equation \eqref{eq:3.7} in the %constructed functional space.

  \textbf{Step 2.} We remove the additional assumptions imposed in Step 1.

Since the coefficients $a$ and $\sigma$ are independent of the spatial variable $x$, BSPDIE \eqref{eq:3.2} can be rewritten in a divergence form
\begin{equation}\label{eq:3.222}
\begin{aligned}
u(t,x) =& \varphi(x) + \int_t^T \Bigg(
      a^{ij}(s)\,D_{ij}u(s,x) + b^i(s,x)\,D_{i}u(s,x) + c(s,x)\,u(s,x) \\
&\ \ \ \ \ \ \ \ \ \ \ \ \ \ \ + \sigma^i(s)\,D_{i}q(s,x) + v(s,x)\,q(s,x) + f_0(s,x) \\
&\ \ \ \ \ \ \ \ \ \ \ \ \ \ \ + \int_{Z}\left[u(s,x+k(s,e,x)) - u(s,x) - k^i(s,e,x)\,D_{i}u(s,x)\right] \nu(de) \\
&\ \ \ \ \ \ \ \ \ \ \ \ \ \ \ + \int_Z \left[ r(s,e,x+k(s,e,x)) - r(s,e,x) \right] \nu(de)
\Bigg) ds \\
& - \int_t^T q(s,x)dW_s - \int_t^T \int_Z r(s,e,x)\tilde{\mu}(de,ds).
\end{aligned}
\end{equation}
By Theorem \ref{thm:3.1}, BSPDIE \eqref{eq:3.222} admits a unique weak solution
$(u, q, r) \in \mathbb{S}^2(H^0) \cap \mathbb{H}^1 \times \mathbb{H}^0 \times \mathbb{H}^{\nu,0}.$

Define
\begin{equation*}
\label{eq:tilde-F}
\begin{aligned}
\tilde{F}(t,x) :=\ & b^i(t,x)\,D_{i} u(t,x) + c(t,x)\,u(t,x) + v(t,x)\,q(t,x) + f_0(t,x) \\
& + \int_Z \left[ u(t,x+k(t,e,x)) - u(t,x) - k^i(t,e,x)\,D_{i} u(t,x) \right] \nu(de) \\
& + \int_Z \left[ r(t,e,x+k(t,e,x)) - r(t,e,x) \right] \nu(de).
\end{aligned}
\end{equation*}
This $\tilde{F}$ is completely determined by the weak solution $(u,q,r)$, and obviously belongs to $\mathbb{H}^0$. Then BSPDIE \eqref{eq:3.222} can be regarded as an equation without all lower-order coefficients
\begin{equation}
\label{eq:tilde-eq}
\begin{aligned}
\tilde{u}(t,x) =&\varphi(x) + \int_t^T  \bigg[ a^{ij}(s)D_{ij} \tilde{u}(s,x) + \sigma^i(s)D_{i} \tilde{q}(s,x) + \tilde{F}(s,x) \bigg] ds \\
& - \int_t^T\tilde{q}(s,x)dW_s - \int_t^T\int_Z \tilde{r}(s,e,x)\tilde{\mu}(de,ds).
\end{aligned}
\end{equation}
By the result in Step 1, BSPDIE \eqref{eq:tilde-eq} admits a unique strong solution $(\tilde{u}, \tilde{q}, \tilde{r})$ in $\mathbb{S}^2(H^1) \cap \mathbb{H}^2 \times \mathbb{H}^1 \times \mathbb{H}^{\nu,1}$, and \eqref{eq:3.18} holds for $(\tilde{u}, \tilde{q}, \tilde{r})$ with $f_0$ replaced by $\tilde{F}$.

Moreover, $(\tilde{u}, \tilde{q}, \tilde{r})$ is also a weak solution to BSPDIE \eqref{eq:tilde-eq}. But we have known that $(u, q, r)$ is a weak solution to \eqref{eq:3.222}, and hence a weak solution to BSPDIE \eqref{eq:tilde-eq}. By the uniqueness of weak solutions to BSPDIE \eqref{eq:tilde-eq}, we conclude that $(\tilde{u}, \tilde{q}, \tilde{r}) = (u, q, r)$ a.s., and the triple $(u, q, r)$ is indeed the unique strong solution to BSPDIE \eqref{eq:3.222}, or equivalently BSPDIE \eqref{eq:3.2} with \( a, \sigma \) independent of \( x \). Consequently, \eqref{eq:3.18} holds for $(u, q, r)$ but with $f_0$ replaced by $\tilde{F}$. Noticing the uniform boundedness of \( b, c, v, k \) and Assumption \ref{ass:b2} (iii), we get from the definition of $\tilde{F}$ that $|||\tilde{F} |||_0^2\leq ||| f_0 |||_0^2$, which implies that \eqref{eq:3.18} holds.
\end{proof}

Next we prove a perturbation result.
\begin{lem}\label{lem:d3}
   Assume that Assumptions \ref{ass:b1}--\ref{ass:b2} hold, the coefficients $b,c,v,k = 0$, and there exists some constant $\epsilon > 0$ such that for any $(\omega,t,x) \in \Omega \times [0,T] \times \mathbb{R}^d$,
    \[
    |a(t,x) - a_0(t)| \leq \epsilon\ \ \ {\text and}\ \ \ |\sigma(t,x) - \sigma_0(t)| \leq \epsilon,
    \]
    where $a_0$ and $\sigma_0$ are independent of \( x \) and satisfy Assumptions \ref{ass:b1}--\ref{mz4}. If $f_0 \in \mathbb{H}^0$ and $\varphi \in L^2_{\mathscr{F}_T}(\Omega; H^1)$, then there exists a constant $\delta > 0$ depending on $\kappa, \alpha, T$ such that when $\epsilon \leq \delta$, BSPDIE \eqref{eq:3.2} admits a unique strong solution \((u, q, r) \in \mathbb{S}^2({H}^1) \cap \mathbb{H}^2 \times \mathbb{H}^1 \times \mathbb{H}^{\nu,1}\), and the estimate \eqref{eq:3.18} holds.
\end{lem}
\begin{proof}
    For given $(\bar{u}, \bar{q}, \bar{r}) \in \mathbb{S}^2(H^1) \cap \mathbb{H}^2 \times \mathbb{H}^1 \times \mathbb{H}^{\nu,1}$, consider BSPDIE
    \begin{equation}\label{eq:d2}
\begin{aligned}
u(t,x)=& \varphi(x) + \int_t^T \Big[ a_0^{ij}(s) D_{ij} u(s,x) + \sigma_0^i(s) D_{i} q(s,x)+ \big(a^{ij}(s,x) - a_0^{ij}(s)\big) D_{ij} \bar{u}(s,x) \\
& \ \ \ \ \ \ \ \ \ \ \ \ \ \ \ \ + \big(\sigma^i(s,x) - \sigma_0^i(s)\big) D_{i} \bar{q}(s,x) + f_0(s,x) \Big]ds \\
&- \int_t^T q(s,x)dW_s - \int_t^T \int_Z r(s,e,x)\tilde{\mu}(de,ds).
\end{aligned}
\end{equation}
By Proposition \ref{prop:c1}, the above BSPDIE admits a unique strong solution $(u,q,r)\in \mathbb{S}^2(H^1) \cap \mathbb{H}^2 \times \mathbb{H}^1 \times \mathbb{H}^{\nu,1}$. Construct a mapping $\Psi$ from $\mathbb{S}^2(H^1) \cap \mathbb{H}^2 \times \mathbb{H}^1 \times \mathbb{H}^{\nu,1}$ into itself by defining $\Psi(\bar{u}, \bar{q}, \bar{r}) := (u, q, r)$.
For $(\bar{u}_i, \bar{q}_i, \bar{r}_i) \in \mathbb{S}^2 (H^1)\cap \mathbb{H}^2 \times \mathbb{H}^1 \times \mathbb{H}^{\nu,1}$, set $(u_i, q_i, r_i)= \Psi(\bar{u}_i, \bar{q}_i, \bar{r}_i)$, $i = 1, 2$. Then applying the estimate \eqref{eq:3.18} to BSPDIE \eqref{eq:d2}, we have
    \[
    E\left[\sup_{0 \leq t \leq T} \| u_1(t) - u_2(t) \|_{H^1}^2 \right] + ||| u_1 - u_2 |||_2^2 + ||| q_1 - q_2 |||_1^2 + ||| r_1 - r_2 |||_{\nu,1}^2
    \]
    \[
    \leq \epsilon C \left( E\left[\sup_{0 \leq t \leq T} \| \bar{u}_1(t) - \bar{u}_2(t) \|_{H^1}^2 \right] + ||| \bar{u}_1 - \bar{u}_2 |||_2^2 + ||| \bar{q}_1 - \bar{q}_2 |||_1^2 + ||| \bar{r}_1 - \bar{r}_2 |||_{\nu,1}^2 \right),
    \]
    where $C = C(\kappa, \alpha, T)$ is a constant.
Take $\delta= \frac{1}{2C}$, and then $\Psi$ is a contraction mapping on the Banach space $\mathbb{S}^2(H^1) \cap \mathbb{H}^2 \times \mathbb{H}^1 \times \mathbb{H}^{\nu,1}$. By the fixed point theorem, it follows that BSPDIE \eqref{eq:3.2} admits a unique strong solution
$(u, q, r)\in \mathbb{S}^2(H^1) \cap \mathbb{H}^2 \times \mathbb{H}^1 \times \mathbb{H}^{\nu,1}$.

Then applying the estimate \eqref{eq:3.18} again to the BSPDIE \eqref{eq:d2} with $(\bar u,\bar q,\bar r)$ replaced by $(u,q,r)$, we have
\begin{align*}
&\mathbb{E}\left[\sup_{0 \leq t \leq T} \|u\|_{H^1}^2 + \int_0^T \left( \|u\|_{H^2}^2 + \|q\|_{H^1}^2 \right) dt + \int_0^T \int_Z \|r\|_{H^1}^2 \nu(de) dt \right] \\
\leq &\epsilon C\mathbb{E}\left[\sup_{0 \leq t \leq T} \|u\|_{H^1}^2 + \int_0^T \left( \|u\|_{H^2}^2 + \|q\|_{H^1}^2 \right) dt + \int_0^T \int_Z \|r\|_{H^1}^2 \nu(de) dt \right] \\
&+ C\left( \mathbb{E} \int_0^T \|f_0\|_{L^2}^2 dt + \mathbb{E} \|\varphi\|_{H^1}^2 \right).
\end{align*}
Noticing $\epsilon C (2C)^{-1}<1$,  we get the desired a priori estimate \eqref{eq:3.18} for the strong solution $(u, q, r)$.
\end{proof}

Then we prove the a priori estimate for the general form of BSPDIE \eqref{eq:3.2} under the assumption on the existence of strong solution.
\begin{lem}\label{lem:d4}
Assume that the conditions of Theorem \ref{thm:b2} hold. If \((u, q, r) \in \mathbb{S}^2(H^1) \cap \mathbb{H}^2 \times \mathbb{H}^1\times \mathbb{H}^{\nu,1}\) is a strong solution of BSPDIE \eqref{eq:3.2}, then there exists a constant \(C > 0\) depending  on \(\alpha\), \(K\), \(\kappa\), \(T\), and the function \(\gamma\) in Assumption \ref{ass:b3}, such that
\begin{equation}\label{eq:3.29}
E\left[\sup_{0 \leq t \leq T} \| u(t) \|_1^2 \right] + ||| u |||_2^2 + ||| q |||_1^2 + ||| r |||_{\nu,1}^2
\leq C \left( ||| f_0 |||_0^2 + E\|\varphi\|_1^2 \right).
\end{equation}

\end{lem}

\begin{proof}
\textbf{Step 1: energy estimate via It\^{o}'s formula.}
 According to the infinite-dimensional It\^{o}'s formula with jumps, we have
\begin{align*}
\|u(t,\cdot)\|^2_0=&\|\phi\|^2 _0+ 2\int_t^T\int_{\mathbb{R}^d}u(s,x)\big[a^{ij}(s,x)D_{ij}u(s,x)
+b^i(s,x)D_{i}u(s,x)+c(s,x)u(s,x)\\
&\ \ \ \ \ \ \ \ \ \ \ \ \ \ \ \ \ \ \ \ \ \ \ \ \ \ \ \ \ \ +\int_Z\left(u(s,x + k(s,e,x))-u(s,x)-k^i(s,e,x)D_{i}u(s,x)\right)\nu(de)\\
&\ \ \ \ \ \ \ \ \ \ \ \ \ \ \ \ \ \ \ \ \ \ \ \ \ \ \ \ \ \ +\sigma^i(s,x)D_{i}q(s,x)+v(s,x)q(s,x)+f_0(s,x)\\
&\ \ \ \ \ \ \ \ \ \ \ \ \ \ \ \ \ \ \ \ \ \ \ \ \ \ \ \ \ \ +\int_Z\left(r(s,e,x + k(s,e,x))-r(s,e,x)\right)\nu(de)\big]dxds\\
&-\int_t^T\int_{\mathbb{R}^d}|q(s,x)|^2dxds- 2\int_t^T\int_{\mathbb{R}^d}u(s,x)q(s,x)dx dW_s\\
&-\int_t^T\int_Z\int_{\mathbb{R}^d}\left(- 2u(s,x)r(s,e,x)-\|r(s,e,x)\|^2\right)dx\tilde{\mu}(de,ds)\\
&-\int_t^T\int_Z\int_{\mathbb{R}^d}|r(s,e,x)|^2dx\nu(de)ds.
\end{align*}
Taking the expectation of both sides and applying Cauchy-Schwarz inequality, we further have
\begin{eqnarray}\label{eq:energy-ineq}
  \begin{split}
  &\|u(t,\cdot)\|^2_0+||| q |||_0^2 + ||| r |||_{\nu,0}^2\\
\leq& \|\phi\|_0^2 + 2\mathbb{E} \big[ \int_t^T \int_{\mathbb{R}^d} u(s,x) \bigg( a^{ij}(s,x)D_{ij}u(s,x) + b^i(s,x)D_{i} u(s,x)+ c(s,x) u(s,x) \\
&\ \ \ \ \ \ \ \ \ \ \ \ \ \ \ \ \ \ \ \ \ \ \ \ \ + \int_Z \left( u(s,x + k(s,e,x)) - u(s,x) - k^i(s,e,x)D_{i} u(s,x) \right) \nu(de)\\
&\ \ \ \ \ \ \ \ \ \ \ \ \ \ \ \ \ \ \ \ \ \ \ \ \ + \sigma^i(s,x)D_{i} q(s,x)+ v(s,x) q(s,x) + f_0(s,x) \\
&\ \ \ \ \ \ \ \ \ \ \ \ \ \ \ \ \ \ \ \ \ \ \ \ \ + \int_Z \left( r(s,e,x + k(s,e,x)) - r(s,e,x) \right) \nu(de) \bigg) dx\, ds \big]\\
\leq\ & \|\phi\|_0^2 + \varepsilon \left( ||| u |||_2^2 + ||| q |||_1^2 + ||| r |||_{\nu,0}^2 \right) + C(\varepsilon, K)\, ||| u |||_0^2 + ||| f_0 |||_0^2,
  \end{split}
\end{eqnarray}
where \(\varepsilon\) is a small positive number to be specified later. %, and \(C(\varepsilon,K)\) is a constant depending on \(\varepsilon\) and some parameters \(K\).

\noindent\textbf{Step 2: localization and regularization.}

Due to Assumption \ref{ass:b3}, we can choose a sufficiently small constant $\rho>0$ such that for any $(\omega,t,x,y)\in\Omega\times[0,T]\times\mathbb{R}^d\times\mathbb{R}^d$ with $|x-y|\leq4\rho$,
\[
|a(t,x)-a(t,y)|\leq\delta,\quad |\sigma(t,x)-\sigma(t,y)|\leq\delta,
\]
where $\delta$ is the constant in Lemma \ref{lem:d3} depending on $\kappa,\alpha,T$.

For any fixed point $z\in\mathbb{R}^d$, denote the open ball
\[
B_l(z):=\big\{x\in\mathbb{R}^d: |x-z|<l\big\}.
\]
Choose a non-negative cut-off function $\zeta\in C_0^\infty(\mathbb{R}^d)$ satisfying
\[
\operatorname{supp}(\zeta)\subset B_{2\rho}(0)\ \ \ \text{and}\ \ \ \zeta(x)=1\ \ \text{for}\ |x|\leq\rho.
\]
Define the localized cut-off functions
\[
\zeta^z(x):=\zeta(x-z)\ \ \ \text{and}\ \ \ \eta^z(x):=\zeta\left(\frac{x-z}{\rho}\right),
\]
and the localized solution
\[
u^z:=u\zeta^z,\quad q^z:=q\zeta^z,\quad r^z:=r\zeta^z.
\]
By It\^o's formula, it is not difficult to verify that $(u^z, q^z, r^z)$ satisfies a localized BSPDIE
\begin{align}\label{eq:local-BSPDIE}
\mathrm{d} u^z(t,x)
= -\Big[\tilde{a}^{ij}(t,x)D_{ij}u^z
+ \tilde{\sigma}^i(t,x)D_{i}q^z
+ \tilde{F}(t,x)\Big]\mathrm{d} t+ q^z(t,x)\mathrm{d}W_t
+ \int_Z r^z(t,e,x)\tilde{\mu}(\mathrm{d} e,\mathrm{d} t),\nonumber
\end{align}
where the coefficients ahead of highest-order terms are defined as
\[
\begin{aligned}
\tilde{a}^{ij}(t,x) &:= a^{ij}(t,x)\eta^z(x) + a^{ij}(t,z)\big(1-\eta^z(x)\big),\\
\tilde{\sigma}^i(t,x) &:= \sigma^i(t,x)\eta^z(x) + \sigma^i(t,z)\big(1-\eta^z(x)\big),
\end{aligned}
\]
and the inhomogeneous term $\tilde{F}$ is defined as
\begin{align*}
\tilde{F}(t,x)
=\, & f_0(t,x)\zeta^z(x) + \Big( b^i(t,x)\zeta^z(x)-2a^{ij}(t,x)D_{j}\zeta^z(x) \Big)D_{i}u(t,x)\\
&+ \Big( v^k(t,x)\zeta^z(x)-\sigma^{ik}(t,x)D_{i}\zeta^z(x) \Big)q^k(t,x)- \Big( c(t,x)\zeta^z(x)+a^{ij}(t,x)D_{ij}\zeta^z(x) \Big)u(t,x)\\
&+ \int_Z  \zeta^z(x) \Big[ r(t,e,x+k(t,e,x))-r(t,e,x) \Big]\nu(\mathrm{d} e)\\
&+ \int_Z\zeta^z(x) \Big[ u(t,x+k(t,e,x))-u(t,x)-k^i(t,e,x)D_{i}u(t,x) \Big]\nu(\mathrm{d} e).
\end{align*}
Then all conditions of Lemma \ref{lem:d3} are satisfied with $a_0(t)=a^{ij}(t,z)$ and $\sigma_0(t)=\sigma^i(t,z)$, due to the choice of sufficiently small $\rho$ and the fact that $\tilde{F}\in\mathbb{H}^0$.

By Lemma \ref{lem:d3}, the localized BSPDIE admits a unique strong solution
\[
(u^z,q^z,r^z)\in \big(\mathbb{S}^2(H^1)\cap\mathbb{H}^2\big)\times \mathbb{H}^1\times \mathbb{H}^{\nu,1},
\]
and it follows from the uniqueness of solutions that
\[
u^z=u,\quad q^z=q\quad\text{and}\quad r^z=r\quad \text{on}\ B_{2\rho}(z).
\]

Moreover, the a priori estimate in Lemma \ref{lem:d3} leads to
\begin{align*}\label{eq:local-est}
&\mathbb{E}\int_0^T \Big(
\|u(t,\cdot)\|_{2,B_\rho(z)}^2
+ \|q(t,\cdot)\|_{1,B_\rho(z)}^2
+ \|r(t,\cdot)\|_{\nu,1,B_\rho(z)}^2
\Big)\mathrm{d} t+ \mathbb{E}\sup_{t\leq T}\|u(t,\cdot)\|_{1,B_\rho(z)}^2\nonumber\\
\leq& C\Bigg(
\mathbb{E}\|\varphi\|_{2,B_{2\rho}(z)}^2
+ \mathbb{E}\int_0^T \|\tilde{F}(t,\cdot)\|_{0,B_{2\rho}(z)}^2\mathrm{d} t
\Bigg).
\end{align*}
Using a standard finite covering argument for $\mathbb{R}^d$ with balls $\{B_\rho(z)\}_{z\in\mathbb{R}^d}$ and integrating the above local estimates for all $z\in\mathbb{R}^d$, we derive the global a priori estimate
\[
\begin{aligned}
&\left|\left|\left|u\right|\right|\right|_2^2 + \left|\left|\left|q\right|\right|\right|_1^2 + \left|\left|\left|r\right|\right|\right|_{\nu,1}^2+ \mathbb{E}\Big[\sup_{0\leq t\leq T}\|u(t)\|_{1}^2\Big]\\
\leq& C\Big(
\left|\left|\left|f_0\right|\right|\right|_0^2 + \mathbb{E}\|\varphi\|_1^2
+ \left|\left|\left|u\right|\right|\right|_1^2 + \left|\left|\left|q\right|\right|\right|_0^2 + \left|\left|\left|r\right|\right|\right|_{\nu,0}^2
\Big).
\end{aligned}
\]
The above estimate, together with the preliminary energy inequality \eqref{eq:energy-ineq} in Step 1, implies that
\begin{align}\label{mz1}
\left|\left|\left|u\right|\right|\right|_2^2 + \left|\left|\left|q\right|\right|\right|_1^2 + \left|\left|\left|r\right|\right|\right|_{\nu,1}^2+ \mathbb{E}\Big[\sup_{0\leq t\leq T}\|u(t)\|_{1}^2\Big]\leq C\Big(
\left|\left|\left|f_0\right|\right|\right|_0^2 + \mathbb{E}\|\varphi\|_1^2
+ \left|\left|\left|u\right|\right|\right|_1^2\Big).
\end{align}
Note that for $t\in[0,T)$, we know $\mathbb{E}\|u(t)\|_{1}^2\leq C\Big(
\left|\left|\left|f_0\right|\right|\right|_0^2 + \mathbb{E}\|\varphi\|_1^2
+ \int_t^T\mathbb{E}\left|\left|u(s)\right|\right|_1^2ds\Big)$. An application of Gronwall's inequality leads to
$\left|\left|\left|u\right|\right|\right|_1^2\leq C\Big(
\left|\left|\left|f_0\right|\right|\right|_0^2 + \mathbb{E}\|\varphi\|_1^2\Big)$.
Consequently, \eqref{eq:3.29} follows immediately from \eqref{mz1}.
\end{proof}

%\section*{Proof of Theorem \ref{thm:b2}}
Now we are ready to prove Theorem \ref{thm:b2}.
\begin{proof}[Proof of Theorem \ref{thm:b2}]
The uniqueness of the strong solution to BSPDIE \eqref{eq:3.2} follows immediately from the a priori estimate \eqref{eq:3.29}. %, which implies that if there exist two solutions, their difference must vanish almost surely.

We turn to the proof of existence by the method of continuation. First we introduce some operators which interpolate between a simplified version of the original BSPDIE and the concerned equation. For this, define
%the linear second-order differential operators \(L_0\) and \(L_1\), the first-order stochastic operators %\(B_0\) and \(B_1\), and the nonlocal jump operator \(M\) as follows:
\begin{eqnarray}\label{mz2}
\left\{
\begin{aligned}
L_0 u=& a^{ij}(t,0)D_{ij} u + b^i(t,0) D_{i} u + c(t,0) u \\
&+ \int_Z \left[ u(t,x+k(t,e,0)) - u - k^i(t,e,0) D_{i} u \right] v(de), \\
L_1 u=& a^{ij}(t,x)D_{ij} u + b^i(t,x) D_{i} u + c(t,x) u\\
 &+ \int_Z \left[ u(t,x+k(t,e,0)) - u - k^i(t,e,0) D_{i} u \right] v(de), \\
B_0 q=& \sigma^i(t,0)D_{i} q + v(t,0) q, \\
B_1 q=& \sigma^i(t,x)D_{i} q + v(t,x) q, \\
M r=& \int_Z \left[ r(t,e,x+k(t,e,x)) - r(t,e,x) \right] \nu(de).
\end{aligned}
\right.
\end{eqnarray}
Then for a parameter \(\lambda \in [0,1]\), the interpolated operators are introduced as follows
\[
L_\lambda := (1-\lambda)L_0 + \lambda L_1\ \ \ \text{and}\ \ \ B_\lambda := (1-\lambda)B_0 + \lambda B_1.
\]
Consider a parameterized BSPDIE
\[
\begin{cases}
du = -\left( L_\lambda u + B_\lambda q + M r + f_0 \right)dt + qdW_t + \displaystyle\int_Z r\,\tilde{\mu}(de,dt), \\
u|_{t=T} = \varphi.
\end{cases}
\]
Since for any \(\lambda \in [0,1]\), the coefficients satisfy the conditions of Theorem \ref{thm:b2} with unified parameters,
the a priori estimate \eqref{eq:3.29} holds with the same constant \(C > 0\) independent of \(\lambda\).

Suppose that the above parameterized BSPDIE admits a unique strong solution for some \(\lambda_0 \in [0,1]\). For any \(\lambda\) close to \(\lambda_0\), we can rewrite the parameterized BSPDIE in a form like
\[
\begin{cases}
du = -\big[ L_{\lambda_0} u + B_{\lambda_0} q + M r + f_0 + (\lambda - \lambda_0)\left( (L_1 - L_0)u + (B_1 - B_0)q \right) \big] dt \\
\qquad\  + qdW_t + \displaystyle\int_Z r\,\tilde{\mu}(de,dt), \\
u|_{t=T} = \varphi.
\end{cases}
\]
For a given triple \((\bar{u}, \bar{q}, \bar{r}) \in \mathbb{S}^2(H^1) \cap \mathbb{H}^2 \times \mathbb{H}^1 \times \mathbb{H}^{\nu,1}\), define a mapping \(\Psi: (\bar{u}, \bar{q}, \bar{r}) \mapsto (u, q, r)\) by setting \((u, q, r)\) to be the strong solution to BSPDIE
\[
\begin{cases}
du = -\left[ L_{\lambda_0} u + B_{\lambda_0} q + M r + f_0 + (\lambda - \lambda_0)\left( (L_1 - L_0)\bar{u} + (B_1 - B_0)\bar{q} \right) \right]dt \\
\qquad\ + qdW_t + \displaystyle\int_Z r\,\tilde{\mu}(de,dt), \\
u|_{t=T} = \varphi.
\end{cases}
\]
By the a priori estimate \eqref{eq:3.29} again, together with the boundedness of operators in \eqref{mz2}, for any \((\bar{u}_1, \bar{q}_1, \bar{r}_1), (\bar{u}_2, \bar{q}_2, \bar{r}_2)\in\mathbb{S}^2(H^1) \cap \mathbb{H}^2 \times \mathbb{H}^1 \times \mathbb{H}^{\nu,1}\), we have
\[
\begin{aligned}
&||| \Psi(\bar{u}_1 - \bar{u}_2) |||_2^2 + ||| \Psi(\bar{q}_1 - \bar{q}_2) |||_1^2 + ||| \Psi(\bar{r}_1 - \bar{r}_2) |||_{\nu,1}^2 + \mathbb{E} \sup_{t \leq T} \|\Psi(\bar{u}_1 - \bar{u}_2)(t)\|_1^2 \\
\leq& C |\lambda - \lambda_0|^2 \left( ||| \bar{u}_1 - \bar{u}_2 |||_2^2 + ||| \bar{q}_1 - \bar{q}_2 |||_1^2 + ||| \bar{r}_1 - \bar{r}_2 |||_{\nu,1}^2 \right),
\end{aligned}
\]
where \(C\) is a positive constant depending on given parameters. Hence, the mapping \(\Psi\) becomes a contraction in the space $\mathbb{S}^2(H^1) \cap \mathbb{H}^2 \times \mathbb{H}^1 \times \mathbb{H}^{\nu,1}$ if we take sufficiently small \(|\lambda - \lambda_0|\). In particular, this mapping is contractive if \(|\lambda - \lambda_0| \leq \frac{1}{2C}\).

By the fixed-point theorem, the equation is solvable in a neighborhood of \(\lambda_0\). Since the case \(\lambda = 0\) corresponds to a linear BSPDIE with frozen coefficients, whose existence of strong solution is already guaranteed by Proposition \ref{prop:c1}, we can initiate the continuation from \(\lambda = 0\). By dividing the interval \([0,1]\) into finitely many subintervals with a size less than $\frac{1}{2C}$, we iteratively construct a sequence of strong solutions for finitely many \(\lambda \in [0,1]\) up to $\lambda=1$. Actually, BSPDIE \eqref{eq:3.2} appears as $\lambda=1$. In this way, the existence of strong solution to BSPDIE \eqref{eq:3.2} is proved.
\end{proof}

Furthermore, as the regularity of the coefficients increases, the regularity of the solution increases correspondingly.
\begin{thm}\label{thm:3.6}
Assume that Assumptions \ref{ass:b1}--\ref{ass:b3} hold. If for a nonnegative integer $n$ and any multi-index $\alpha$ such that $|\alpha| \leq n$, we have
  \begin{equation*}\label{con:b1}
\begin{aligned}
    \mathop{\mathrm{ess}\sup}\limits_{\Omega\times[0,T]\times \mathbb{R}^d} &\left( |D^{\alpha}a| + |D^{\alpha}b| + |D^{\alpha}c| + |D^{\alpha}\sigma| + |D^{\alpha}\nu| \right)+\mathop{\mathrm{ess\,sup}}\limits_{\Omega\times[0,T]\times Z \times \mathbb{R}^d} |D^\alpha k(x,e)| \leq L,\\
    &\ \ \ \ \ \ \ \ \ \ \ \ \ f_0\in  \mathbb{H}^n\ \ \text{and}\ \ \varphi \in L^2(\Omega, \mathscr{F}_T; H^{n+1}),
\end{aligned}
  \end{equation*}
then BSPDIE \eqref{eq:3.2} admits a unique strong solution $(u, q, r)\in \mathbb{S}^2(H^{n+1}) \cap \mathbb{H}^{n+2}\times \mathbb{H}^{n+1}\times\mathbb{H}^{\nu,n+1}$ satisfying
  \begin{equation}\label{eq:3.33}
    \begin{split}
    \mathbb{E}\left[\sup_{0 \leq t \leq T} \| u(t) \|_{n+1}^2 \right] + ||| u |||_{n+2}^2 + ||| q |||_{n+1}^2 + ||| r |||_{\nu, n+1}^2\leq C \left( ||| f_0 |||_{n}^2 + \mathbb{E} \| \varphi \|_{n+1}^2 \right),
    \end{split}
  \end{equation}
  where \( {{C = C(L,\kappa,\alpha,\delta,T)}} \) is a positive constant.
\end{thm}
\begin{proof}
We prove the theorem by induction for the regularity index $n$.

\textbf{Base Case ($n = 0$).} The conditions of Theorem \ref{thm:b2} are satisfied. BSPDIE \eqref{eq:3.2} admits a unique strong solution $(u, q, r) \in \mathbb{S}^2(H^1) \cap \mathbb{H}^2 \times \mathbb{H}^1 \times \mathbb{H}^{\nu,1}$ satisfying the estimate \eqref{eq:3.18} which is actually \eqref{eq:3.33} with $n=0$.

\textbf{Inductive Step.} Assume Theorem \ref{thm:3.6} is true for some $n = k \geq 0$, and hence BSPDIE \eqref{eq:3.2} admits a unique strong solution $(u, q, r) \in \mathbb{S}^2(H^{k+1}) \cap \mathbb{H}^{k+2} \times \mathbb{H}^{k+1} \times \mathbb{H}^{\nu,k+1}$ satisfying the estimate \eqref{eq:3.33}. We prove that it is still true for $n = k+1$ if the conditions of Theorem \ref{thm:b2} hold for  $|\alpha| \leq k+1$.

Take the derivative of spatial variable on both sides of BSPDIE \eqref{eq:3.2}. Since the coefficients have bounded derivatives up to order $k+1$, the interchange of the derivative operator $D^\alpha$ and the integration is permissible. Define new variables
\[
\bar{u} := D^\alpha u, \quad \bar{q} := D^\alpha q, \quad \bar{r} := D^\alpha r.
\]
Then $(\bar{u}, \bar{q}, \bar{r})$ satisfies BSPDIE
\begin{equation}\label{eq:b5}
\begin{aligned}
\bar{u}(t,x) =&D^\alpha \varphi(x) + \int_t^T  \bigg[ a^{ij}(s,x) D_{ij} \bar{u}(s,x) + \sigma^i(s,x) D_i \bar{q}(s,x) + \tilde{f}_0(s,x)  \bigg] ds \\
& - \int_t^T\bar{q}(s,x)dW_s - \int_t^T\int_Z \bar{r}(s,e,x)\tilde{\mu}(de,ds),
\end{aligned}
\end{equation}
where
\[
\begin{aligned}
\tilde{f}_0 ={}& D^\alpha f_0 + \sum_{\substack{|\beta|+|\gamma|=|\alpha| \\ |\beta|\geq 1}} C_{\alpha\beta} \left[(D^\beta a^{ij})(D^\gamma u_{x^i x^j}) + (D^\beta \sigma^i)(D^\gamma q_{x^i})\right] \\
&+ \sum_{|\beta|+|\gamma|=|\alpha|} C_{\alpha\beta} \left[(D^\beta b^i)(D^\gamma u_{x^i}) + (D^\beta c)(D^\gamma u) + (D^\beta \nu)(D^\gamma q)\right] \\
& - \int_Z D^\alpha u \, \nu(de) - \int_Z \sum_{|\beta|+|\gamma|=|\alpha|} C_{\alpha\beta} (D^\beta k^i(s,e,x))(D^\gamma u_{x^i}) \nu(de) \\
&+ \int_Z D^\alpha \left[ r(s,e, x + k(s,e,x)) - r(s, e, x) \right] \nu(de)+ \int_Z D^\alpha \left[u(s,x + k(s,e,x))\right] \nu(de).
\end{aligned}
\]
Note that  $\tilde{f}_0 \in \mathbb{H}^0$ and satisfies the estimate
\begin{align*}
||| \tilde{f}_0 |||_0^2 &\leq C(K, \kappa, T) \left( ||| D^\alpha f_0 |||_0^2 + ||| u |||_{k+1}^2 + ||| q |||_{k+1}^2 + ||| r |||_{\nu,k+1}^2 \right) \\
&\leq C(K, \kappa, T) \left( ||| f_0 |||_{k+1}^2 + \mathbb{E} \| \varphi \|_{k+1}^2 \right),
\end{align*}
where we used the regularity assumptions on $f_0$ and $\varphi$, and the induction hypothesis in the second inequality.

Now, applying Theorem \ref{thm:b2} to the BSPDIE \eqref{eq:b5}, we have
\[
(\bar{u}, \bar{q}, \bar{r}) \in \mathbb{S}^2(H^1) \cap \mathbb{H}^2, \times \mathbb{H}^1\times \mathbb{H}^{\nu,1}.
\]
In particular, for the multi-index $\alpha$ with $|\alpha| = k+1$, it follows that
\[
({u}, {q}, {r}) \in \mathbb{S}^2(H^{k+2}) \cap \mathbb{H}^{k+3}, \times \mathbb{H}^{k+2}\times \mathbb{H}^{\nu,k+2},
\]
i.e., the solution has the desired regularity for $n = k+1$. Moreover, the estimate \eqref{eq:3.33} follows from the corresponding estimate in Theorem \ref{thm:b2} with $f_0=\tilde{f}_0$.

By mathematical induction, Theorem \ref{thm:3.6} holds for all nonnegative integers.
\end{proof}

\begin{rmk}
When the coefficients are sufficiently smooth, it is natural to expect higher Sobolev regularity of the solution. In particular, for $n$ large enough such that $n > \frac{d}{2}$, the Sobolev embedding $H^{n+2}\subset C^{2,\gamma}$ (for some $\gamma\in(0,1)$) would imply that the strong solution is a classical solution.
%However, a rigorous proof for higher-order regularity of the nonlocal jump term $u(x+k(t,e,x))$ requires a %full Fa\`a di Bruno expansion and commutator estimates, and the method of \cite{ch-ta} relying on %stochastic flow techniques are not directly applicable under our mild regularity assumptions. When the %jump amplitude $k$ is independent of $x$, the difficulty disappears and standard differentiation of the %equation yields an induction argument. For the general $x$-dependent case, we leave the full higher-order %regularity analysis for future work.
\end{rmk}

\section{An Application: Comparison Principle for BSPDIE}

%\subsection{Problem Setting}

Comparison principle plays an important role in the theory and applications of partial differential equations and BSDEs. Since BSPDEs combine features of both PDEs and BSDEs, comparison principles are also fundamental for BSPDE. The comparison principle for strong solutions to BSPDEs  was first established in Ma and Yong \cite{ma-yo2}, laying an important foundation for subsequent applications. After that, Du and Meng \cite{du-me} revisited this issue under weaker conditions. Inspired by this existing literature, we aim to establish the comparison principle for strong solutions to BSPDIEs, and the key point is how to deal with the integro-differential structure arising from jump processes.

We consider the comparison principle for BSPDIE \eqref{eq:3.2} with the generator $(\varphi,f_0)$ replaced by $(\varphi_1,f_{01})$ and $(\varphi_2,f_{02})$, respectively. But an additional sign condition on the $0$-order term is needed.
\begin{ass}[Sign Condition]\label{ass:comp-sign}
For any $(\omega,t,x) \in \Omega \times [0,T] \times \mathbb{R}^d$, the zero-order coefficient $c(t,x) \leq 0$.
\end{ass}

\begin{ass}\label{ass:b4}
For any $\omega\in\Omega$ and $t\in[0,T]$, the coefficients $a$, $b$, $c$, $\sigma$, $v$ satisfy the following additional conditions:
\begin{itemize}
  \item $a^{ij}(\omega,t,\cdot) \in W^{2,\infty}(\mathbb{R}^d)$ with $\|D^2 a\|_{L^\infty}$ uniformly bounded;
  \item $b^i(\omega,t,\cdot) \in W^{1,\infty}(\mathbb{R}^d)$ with $\|Db\|_{L^\infty}$ uniformly bounded;
  \item $\sigma^i(\omega,t,\cdot) \in W^{1,\infty}(\mathbb{R}^d)$ with $\|D\sigma\|_{L^\infty}$ uniformly bounded;
  \item $c(\omega,t,\cdot), v(\omega,t,\cdot) \in L^\infty(\mathbb{R}^d)$ uniformly bounded.
\end{itemize}
\end{ass}

\begin{ass}[Measure-preserving jump transformation]\label{ass:measure-preserving}
For each $(\omega,t,e)\in\Omega\times[0,T]\times Z$, the mapping $\Psi_{t,e}(x):=x+k(t,e,x)$ is a $C^1$-diffeomorphism from $\mathbb{R}^d$ onto itself and preserves Lebesgue measure, i.e.,
\[
\big|\det\big(I+\partial_x k(t,e,x)\big)\big| = 1.
\]
and moreover, $\partial_x k$ is uniformly bounded.
\end{ass}

%\subsection{Main Result}

\begin{thm}[Comparison principle for BSPDIE]\label{thm:comp}
Assume that Assumptions \ref{ass:b1}, \ref{ass:b2} (i)(ii), \ref{mz4}, \ref{ass:b3} and \ref{ass:comp-sign}--\ref{ass:measure-preserving} hold, and let $(u_1,q_1,r_1)$ and $(u_2,q_2,r_2)$ be two strong solutions to BSPDIE \eqref{eq:3.2} corresponding to the data $(\varphi_1,f_{01})$ and $(\varphi_2,f_{02})$, respectively, with $f_{01},f_{02}\in\mathbb H^0$ and $\varphi_1,\varphi_2\in L^2_{\mathscr F_T}(\Omega;H^1)$. If
\[
\varphi_1(x)\ge\varphi_2(x)\quad\text{for \ }(\omega,x)\in\Omega\times\mathbb R^d,
\]
and
\[
f_{01}(t,x)\ge f_{02}(t,x)\quad\text{for \ }(\omega,t,x)\in\Omega\times[0,T]\times\mathbb R^d,
\]
then for any $t \in [0,T]$,
\[
u_1(t,x) \geq u_2(t,x) \quad \text{a.e.  a.s.}
\]
\end{thm}

%\subsection{Technical Preparation}

To prove Theorem \ref{thm:comp}, we need the estimate for the negative part of solutions.
\begin{lem}[Negative part estimate]\label{lem:neg-est}
Assume that Assumptions \ref{ass:b1}, \ref{ass:b2} (i)(ii), \ref{mz4}, \ref{ass:b3} and \ref{ass:comp-sign}--\ref{ass:measure-preserving} hold, and let $(u, q, r)$ be a strong solution to BSPDIE \eqref{eq:3.2}. If $f_0 \in \mathbb{H}^0$ and $\varphi\in L^2_{\mathscr{F}_T}(\Omega; H^1)$, then there exists a constant $C > 0$ such that for any $t \in [0,T]$,
\[
\mathbb{E} \int_{\mathbb{R}^d} [u(t,x)^-]^2 dx \leq C e^{C(T-t)} \left(\mathbb{E} \int_{\mathbb{R}^d} [\varphi(x)^-]^2 dx + \mathbb{E} \int_t^T \int_{\mathbb{R}^d} [f_0(s,x)^-]^2 dx ds \right),
\]
where $y^- = \max(-y, 0)$ for $y\in\mathbb{R}$.
\end{lem}

\begin{proof}
All integrations by parts over $\mathbb{R}^d$ in this proof are first justified for smooth and compactly supported functions, for which the boundary terms vanish. The general case follows by a standard density argument in Sobolev spaces.

For $\varepsilon>0$, let $h_\varepsilon\in C^2(\mathbb R)$ be a standard convex regularization of $(x^-)^2$ such that $h_\varepsilon(x)=0$ for $x\ge0$, and $0\le h_\varepsilon(x)\le C(x^-)^2$, $h_\varepsilon'(x)\le0$ and $xh_\varepsilon'(x)\ge0$ for any $x\in\mathbb R$. We choose the approximation such that
\[
h_\varepsilon(x)\to (x^-)^2,\qquad
h_\varepsilon'(x)\to -2x^-,\qquad
0\le h_\varepsilon''(x)\le C,
\]
uniformly in $\varepsilon$, and there exists a constant $C_0>0$ independent of $\varepsilon$ such that for any $x\in\mathbb{R}$,
\[
|h_\varepsilon'(x)|^2\le C_0 h_\varepsilon(x),\ \ \ \text{and}\ \ \
|h_\varepsilon'(x)|^2\le C_0 h_\varepsilon(x)h_\varepsilon''(x).
\]
Such a family is obtained, for instance, by smoothing $(x^-)^2$ only on the interval $[-\varepsilon,0]$ and keeping it equal to $x^2$ on $(-\infty,-\varepsilon]$ and to $0$ on $[0,\infty)$.
We apply It\^o's formula for a.e.\ $x\in\mathbb R^d$ to the composition $h_\varepsilon(u(t,x))$. In view of the regularity of the strong solution $u\in\mathbb{S}^2(H^1)\cap\mathbb{H}^2$, this is justified by first mollifying $u$ in the spatial variable, applying the finite-dimensional It\^o formula with jumps %(see, e.g., \cite{gyo-kry})
to $h_\varepsilon(u_n(t,x))$, and then passing to the limit $u_n\to u$ in $L^2(\Omega\times[0,T]\times\mathbb{R}^d)$. % by the Sobolev convergence, the uniform quadratic/linear growth bounds on $h_\varepsilon,h_\varepsilon'$ and  the uniform bound on $h_\varepsilon''$.
Then it turns out that
\[
\begin{aligned}
d\,h_\varepsilon(u(t,x))
=&-h_\varepsilon'(u(t-,x))\Big[
a^{ij}(t,x)D_{ij}u(t,x)+b^i(t,x)D_i u(t,x)+c(t,x)u(t,x) \\
&\qquad\qquad\qquad\ \
+\int_Z\Big(u(t,x+k(t,e,x))-u(t,x)-k^i(t,e,x)D_i u(t,x)\Big)\nu(de) \\
&\qquad\qquad\qquad\ \
+\sigma^i(t,x)D_i q(t,x)+v(t,x)q(t,x)+f_0(t,x) \\
&\qquad\qquad\qquad\ \
+\int_Z\Big(r(t,e,x+k(t,e,x))-r(t,e,x)\Big)\nu(de)
\Big]dt \\
&+h_\varepsilon'(u(t-,x))\,q(t,x)dW_t
+\frac12 h_\varepsilon''(u(t-,x))|q(t,x)|^2dt \\
&+\int_Z\Big[
h_\varepsilon\big(u(t-,x)+r(t,e,x)\big)-h_\varepsilon(u(t-,x))
-h_\varepsilon'(u(t-,x))r(t,e,x)
\Big]\nu(de)dt \\
&+\int_Z\Big[
h_\varepsilon\big(u(t-,x)+r(t,e,x)\big)-h_\varepsilon(u(t-,x))
\Big]\tilde{\mu}(de,dt).
\end{aligned}
\]
Taking integration over $\Omega\times[t,T]\times\mathbb R^d$, we have
\begin{equation}\label{eq:ito-neg}
\begin{aligned}
&\mathbb E\int_{\mathbb R^d} h_\varepsilon(u)dx
+\frac12\mathbb E\int_t^T\int_{\mathbb R^d} h_\varepsilon''(u)|q|^2dxds \\
&+\mathbb E\int_t^T\int_{\mathbb R^d}\int_Z
\Big[h_\varepsilon(u+r)-h_\varepsilon(u)-h_\varepsilon'(u)r\Big]\nu(de)dxds \\
=&
\mathbb E\int_{\mathbb R^d} h_\varepsilon(\varphi)dx
+J_1+J_2+J_3+J_4+J_5,
\end{aligned}
\end{equation}
where
\[
\begin{aligned}
J_1&=
\mathbb E\int_t^T\int_{\mathbb R^d} h_\varepsilon'(u)\,a^{ij}D_{ij}udxds,\\
J_2&=
\mathbb E\int_t^T\int_{\mathbb R^d} h_\varepsilon'(u)\,b^iD_i udxds,\\
J_3&=
\mathbb E\int_t^T\int_{\mathbb R^d} h_\varepsilon'(u)\,c\,udxds,\\
J_4&=
\mathbb E\int_t^T\int_{\mathbb R^d} h_\varepsilon'(u)\,(\sigma^iD_i q+vq+f_0)dxds,\\
J_5&=
\mathbb E\int_t^T\int_{\mathbb R^d} \int_Z h_\varepsilon'(u)
\big[u(\cdot+k)-u-k^iD_i u+r(\cdot+k)-r\big]\nu(de)dxds.
\end{aligned}
\]
Next we estimate $J_1$--$J_5$ one by one.

For $J_1$, the integration by parts gives
\[
\begin{aligned}
J_1
&=
-\mathbb E\int_t^T\int_{\mathbb R^d}
h_\varepsilon''(u)a^{ij}D_i uD_j udxds
-\mathbb E\int_t^T\int_{\mathbb R^d}
h_\varepsilon'(u)(D_i a^{ij})D_j udxds.
\end{aligned}
\]
For the second term on the right hand side of above equality, using the integration by parts again, we obtain
\[
\begin{aligned}
-\mathbb E\int_t^T\int_{\mathbb R^d} h_\varepsilon'(u)(D_i a^{ij})D_j udxds
=&
-\mathbb E\int_t^T\int_{\mathbb R^d} D_j\big(h_\varepsilon(u)\big)(D_i a^{ij})dxds\\
=&
\mathbb E\int_t^T\int_{\mathbb R^d} h_\varepsilon(u)D_{ij} a^{ij}dxds.
\end{aligned}
\]
Therefore,
\[
J_1
=
-\mathbb E\int_t^T\int_{\mathbb R^d}
h_\varepsilon''(u)a^{ij}D_i uD_j udxds
+
\mathbb E\int_t^T\int_{\mathbb R^d}
h_\varepsilon(u)D_{ij}a^{ij}dxds.
\]
By the uniform ellipticity of $a^{ij}$ and the $W^{2,\infty}$-regularity in Assumption \ref{ass:b4}, we obtain
\[
J_1
\le
-\alpha \mathbb E\int_t^T\int_{\mathbb R^d}
h_\varepsilon''(u)|Du|^2dxds
+
\|D^2 a\|_{L^\infty}\,\mathbb E\int_t^T\int_{\mathbb R^d}
h_\varepsilon(u)dxds.
\]

For $J_2$, since $D_i(h_\varepsilon(u))=h_\varepsilon'(u)D_i u$, we have
\[
\begin{aligned}
J_2
&=
\mathbb E\int_t^T\int_{\mathbb R^d} b^i D_i(h_\varepsilon(u))dxds
=
-\mathbb E\int_t^T\int_{\mathbb R^d}
h_\varepsilon(u) D_i b^idxds.
\end{aligned}
\]
Hence, by the $W^{1,\infty}$-regularity of $b$ in Assumption \ref{ass:b4}, we have
\[
J_2
\le
C\mathbb E\int_t^T\int_{\mathbb R^d} h_\varepsilon(u)dxds.
\]

For $J_3$, since $x h_\varepsilon'(x)\ge 0$ for any $x\in\mathbb{R}$ and $c\le 0$, we have
\[
J_3 = \mathbb{E}\int_t^T\int_{\mathbb{R}^d} h_\varepsilon'(u)\,c\,udxds \le 0.
\]

For $J_4$, the integration by parts yields
\[
\mathbb E\int_{\mathbb R^d} h_\varepsilon'(u)\sigma^iD_i qdx
=
-\mathbb E\int_{\mathbb R^d} h_\varepsilon''(u)\sigma^i D_i u qdx
-
\mathbb E\int_{\mathbb R^d} h_\varepsilon'(u)D_i\sigma^i qdx.
\]
We now estimate each term related to $J_4$ one by one. Let $C_b:=\|D\sigma\|_{L^\infty}+\|v\|_{L^\infty}$.
For the terms involving $D_i\sigma^i$ and $v$, due to $|h_\varepsilon'(x)|^2\le C_0 h_\varepsilon(x)h_\varepsilon''(x)$ for any $x\in\mathbb{R}$ and Young's inequality, for any $\eta>0$,
\[
\begin{aligned}
\bigl|h_\varepsilon'(u)D_i\sigma^i q\bigr|
\le C_b\sqrt{C_0}\,\sqrt{h_\varepsilon(u)}\,\sqrt{h_\varepsilon''(u)}\,|q| \le \frac{\eta}{2}\,h_\varepsilon''(u)|q|^2+\frac{C_b^2 C_0}{2\eta}\,h_\varepsilon(u),
\end{aligned}
\]
and the same bound holds for $h_\varepsilon'(u)v q$. Hence for any $\eta>0$,
\[
\left|\mathbb E\int h_\varepsilon'(u)D_i\sigma^i qdx\right|
+
\left|\mathbb E\int h_\varepsilon'(u)v qdx\right|
\le
\eta\,\mathbb E\int h_\varepsilon''(u)|q|^2dx
+
\frac{C_b^2 C_0}{\eta}\,\mathbb E\int h_\varepsilon(u)dx .
\]
For the crossing term involving $Du$ and $q$, by the fact that $|\sigma|^2\le\kappa^2$ and Young's inequality again, %$2ab\le\alpha a^2+\frac{1}{\alpha}b^2$ with $a=\sqrt{h_\varepsilon''(u)}|Du|$, $b=\sqrt{h_\varepsilon''(u)}|\sigma||q|/\alpha$,
it yields that
\[
\begin{aligned}
-\mathbb E\int h_\varepsilon''(u)\sigma^i D_i u qdx
\le \mathbb E\int h_\varepsilon''(u)|\sigma||Du||q|dx \le \frac{\alpha}{2}\mathbb E\int h_\varepsilon''(u)|Du|^2dx
+\frac{\kappa^2}{2\alpha}\mathbb E\int h_\varepsilon''(u)|q|^2dx.
\end{aligned}
\]
Summing these estimates, we obtain for any $\eta>0$,
\[
\begin{aligned}
J_4
\le&
\frac{\alpha}{2}\mathbb E\int_t^T\int_{\mathbb R^d}
h_\varepsilon''(u)|Du|^2dxds
+
\left(\frac{\kappa^2}{2\alpha}+\eta\right)
\mathbb E\int_t^T\int_{\mathbb R^d}
h_\varepsilon''(u)|q|^2dxds \\
&
+ \frac{C_b^2 C_0}{\eta}\,\mathbb E\int_t^T\int_{\mathbb R^d} h_\varepsilon(u)dxds
+\mathbb E\int_t^T\int_{\mathbb R^d} h_\varepsilon'(u)f_0dxds .
\end{aligned}
\]
Now choose $\eta>0$ small enough so that $\frac{\kappa^2}{2\alpha}+\eta<\frac12$. This is possible since $\kappa<\sqrt{\alpha}/2$ implies $\frac{\kappa^2}{2\alpha}<\frac18$, and we may take e.g.\ $\eta=\frac18$. Then after estimating $J_5$ below, the terms involving $|q|^2$ and $|Du|^2$ on the right hand side of the $J_4$ estimate will be absorbed into the corresponding nonnegative terms on the left hand side of \eqref{eq:ito-neg}. 

Finally, we turn to the jump term $J_5$. For any $(\omega,s,e,x)\in \Omega \times [0,T] \times Z\times\mathbb{R}^d$, set $V(x)=V(\omega,s,e,x):=u(s,x)+r(s,e,x)$. Then
\begin{align*}
u(s,x+k(s,e,x))-u(s,x)+r(s,e,x+k(s,e,x))-r(s,e,x)
=V(x+k(s,e,x))-V(x).
\end{align*}
Define
\[
\mathcal Z_\varepsilon:=
\mathbb E\int_t^T\int_{\mathbb R^d}\int_Z
\Big[h_\varepsilon(u+r)-h_\varepsilon(u)-h_\varepsilon'(u)r\Big]\nu(de)dxds.
\]
A direct calculation gives
\[
\begin{aligned}
J_5-\mathcal Z_\varepsilon
=&
\mathbb E\int_t^T\int_{\mathbb R^d}\int_Z
\Big[
h_\varepsilon'(u)\big(V(x+k(s,e,x))-V(x)-k^i(s,e,x)D_i u\big)
\Big]\nu(de)dxds \\
&-\mathbb E\int_t^T\int_{\mathbb R^d}\int_Z
\Big[
h_\varepsilon(V(x))-h_\varepsilon(u)-h_\varepsilon'(u)(V(x)-u)
\Big]\nu(de)dxds.
\end{aligned}
\]
Since $h_\varepsilon$ is convex, for any $x$,
\[
h_\varepsilon(V(x+k(s,e,x))) \ge h_\varepsilon(u(x)) + h_\varepsilon'(u(x))\big(V(x+k(s,e,x))-u(x)\big).
\]
Consequently,
\[
\begin{aligned}
&\big[h_\varepsilon'(u)\big(V(x+k(s,e,x))-V(x)-k^i(s,e,x)D_i u\big)\big]
-\big[h_\varepsilon(V(x))-h_\varepsilon(u)-h_\varepsilon'(u)(V(x)-u)\big]\\
\le &h_\varepsilon(V(x+k(s,e,x)))-h_\varepsilon(V(x))-k^i(s,e,x)D_i h_\varepsilon(u).
\end{aligned}
\]
Integrating over $\mathbb{R}^d\times Z$, we obtain
\[
J_5-\mathcal Z_\varepsilon \le \mathbb E\int_t^T\int_{\mathbb R^d}\int_Z\Big[ h_\varepsilon(V(x+k(s,e,x)))-h_\varepsilon(V(x))-k^i(s,e,x) D_i\big(h_\varepsilon(u)\big)\Big]\nu(de)dxds.
\]
By Assumption \ref{ass:measure-preserving}, for any $(\omega,s,e)\in \Omega \times [0,T] \times Z$,
\[
\int_{\mathbb{R}^d} h_\varepsilon(V(x+k(s,e,x)))dx = \int_{\mathbb{R}^d} h_\varepsilon(V(x))dx.
\]
Hence
\[
\int_{\mathbb{R}^d}\big[h_\varepsilon(V(x+k(s,e,x)))-h_\varepsilon(V(x))\big]dx = 0.
\]
Thus
\[
J_5-\mathcal Z_\varepsilon \le -\mathbb{E}\int_t^T\int_Z\int_{\mathbb{R}^d} k^i(s,e,x)D_i h_\varepsilon(u(s,x))dx\,\nu(de)ds.
\]
By integration by parts,
\[
-\int_{\mathbb{R}^d} k^i(s,e,x)D_i h_\varepsilon(u(s,x))dx = \int_{\mathbb{R}^d} D_i k^i(s,e,x)\,h_\varepsilon(u(s,x))dx.
\]
Noticing from Assumptions \ref{ass:measure-preserving} that $\partial_xk$ is bounded, we obtain
\[
J_5 \le C\,\mathbb{E}\int_t^T\int_{\mathbb{R}^d} h_\varepsilon(u(s,x))dxds+\mathcal Z_\varepsilon.
\]
The term $\mathcal Z_\varepsilon$ also appears on the left hand side of \eqref{eq:ito-neg}, and hence it cancels in the estimate below.

Combining the estimates from $J_1$ to $J_5$, we have
\begin{equation}\label{mz3}
\begin{aligned}
&\mathbb E\int_{\mathbb R^d} h_\varepsilon(u)dx
+(\frac12-\frac{\kappa^2}{2\alpha}-\eta)\mathbb E\int_t^T\int_{\mathbb R^d} h_\varepsilon''(u)|q|^2dxds
+\frac{\alpha}{2}\mathbb E\int_t^T\int_{\mathbb R^d} h_\varepsilon''(u)|Du|^2dxds \\
\le&
\mathbb E\int_{\mathbb R^d} h_\varepsilon(\varphi(x))dx
+
C\mathbb E\int_t^T\int_{\mathbb R^d} h_\varepsilon(u)dxds
+
\mathbb E\int_t^T\int_{\mathbb R^d} h_\varepsilon'(u)f_0dxds.
\end{aligned}
\end{equation}
Moreover, since $h_\varepsilon'(x)\le 0$ for any $x\in\mathbb R$, we have
\[
h_\varepsilon'(u)f_0
=
h_\varepsilon'(u)f_0^+ - h_\varepsilon'(u)f_0^-
\le
-h_\varepsilon'(u)f_0^-.
\]
By Young's inequality, together with $|h_\varepsilon'(x)|^2\le C_0 h_\varepsilon(x)$,
\[
h_\varepsilon'(u)f_0\leq -h_\varepsilon'(u)f_0^-
\le
\frac12|h_\varepsilon'(u)|^2+\frac12|f_0^-|^2
\le
\frac{C_0}{2} h_\varepsilon(u)+\frac12|f_0^-|^2.
\]
Therefore, it yields from \eqref{mz3} that
\[
\mathbb E\int_{\mathbb R^d} h_\varepsilon(u(t,x))dx
\le
\mathbb E\int_{\mathbb R^d} h_\varepsilon(\varphi(x))dx
+
C\mathbb E\int_t^T
\int_{\mathbb R^d} h_\varepsilon(u(s,x))dxds
+
C\mathbb E\int_t^T\int_{\mathbb R^d}|f_0^-(s,x)|^2dxds.
\]
Taking $\varepsilon\rightarrow0$, by the dominated convergence theorem %$0\le h_\varepsilon(x)\le C(x^-)^2$,
we have
\[
\mathbb E\int_{\mathbb R^d} |u^-(t,x)|^2dx
\le
\mathbb E\int_{\mathbb R^d} |\varphi^-(x)|^2dx
+
C\int_t^T
\mathbb E\int_{\mathbb R^d}|u^-(s,x)|^2dxds
+
C\mathbb E\int_t^T\int_{\mathbb R^d}|f_0^-(s,x)|^2dxds,
\]
which leads to the negative part estimate in Lemma \ref{lem:neg-est} by an application of Gronwall's inequality.
\end{proof}

\begin{rmk}
The main difficulty in the comparison principle arises from the nonlocal jump term,
which cannot be handled by standard local estimates. The key idea is to exploit the
convexity of the test function and combine it with a change-of-variables argument
based on the diffeomorphism property of the jump mapping.
\end{rmk}

Now we are ready to prove the comparison principle.
\begin{proof}[Proof of Theorem \ref{thm:comp}]
Set
\[
\bar{u} = u_1 - u_2,\quad
\bar{q} = q_1 - q_2,\quad
\bar{r} = r_1 - r_2,\quad
\bar{\varphi} = \varphi_1 - \varphi_2,\quad
\bar{f} = f_{01} - f_{02}.
\]
By the linearity of BSPDIE \eqref{eq:3.2}, the triple $(\bar{u},\bar{q},\bar{r})$ satisfies the same equation with the generator $(\bar{\varphi},\bar{f})$. The conditions of Theorem \ref{thm:comp} indicate
\[
\bar{\varphi}(x) \ge 0\ \ \ \text{and} \ \ \ \bar{f}(t,x) \ge 0.
\]
and hence
\[
\bar{\varphi}^- = 0\ \ \ \text{and} \ \ \ \bar{f}^- = 0.
\]
By Lemma \ref{lem:neg-est} we have
\[
\mathbb{E} \int_{\mathbb{R}^d} [\bar{u}(t,x)^-]^2 dx
\le
C e^{C(T-t)}
\left(
\mathbb{E} \int_{\mathbb{R}^d} [\bar{\varphi}(x)^-]^2 dx
+
\mathbb{E} \int_t^T \int_{\mathbb{R}^d} [\bar{f}(s,x)^-]^2 dx ds
\right)
= 0.
\]
Consequently, for any $t\in[0,T]$,
\[
\mathbb{E} \int_{\mathbb{R}^d} [\bar{u}(t,x)^-]^2 dx = 0,
\]
which implies
\[
\bar{u}(t,x)^- = 0 \quad \text{a.e.  a.s. }
\]
Therefore, for any $t\in[0,T]$,
\[
u_1(t,x) - u_2(t,x)=\bar{u}(t,x) \ge 0 \quad \text{a.e.  a.s. }
\]
\end{proof}

\begin{rmk}
The comparison principle is a fundamental tool in the theory of partial differential equations and their stochastic counterparts. It plays a key role in stochastic control by ensuring the uniqueness and monotonicity of value functions. In mathematical finance, it also provides a theoretical basis for the consistency of pricing equations and the absence of arbitrage. %In addition, it serves as an essential framework for studying monotonicity and stability properties in models arising from population dynamics.
\end{rmk}

\end{document}